\documentclass{amsart}
\address{Department of Mathematics, Institute of Science Tokyo, 2-12-1 Ookayama, Meguro-ku, Tokyo, 152-8551, Japan}
\email{tanaka.a.2255@m.isct.ac.jp}

\usepackage{amsmath,amssymb}
\usepackage{amsfonts}
\usepackage{graphicx}

\usepackage{hyperref}
\usepackage{caption}
\usepackage{mathtools} 
\usepackage{enumitem}
\usepackage[margin=10pt]{caption} 

\usepackage[final]{pdfpages}
\usepackage{amsthm}
\usepackage{color}

\usepackage{tikz}    

\theoremstyle{plain}
\newtheorem{thm}{Theorem}[section]

\newtheorem{cor}[thm]{Corollary}

\newtheorem*{thm*}{Theorem}
\newtheorem*{cor*}{Corollary}

\theoremstyle{definition}
\newtheorem{dfn}[thm]{Definition}
\newtheorem{rem}[thm]{Remark}

\newtheorem*{que*}{Question}
\newtheorem*{con*}{Conjecture}

\newcommand{\bn}{\mathrm{bn}}   
\newcommand{\hn}{\mathrm{hn}}   
\newcommand{\gr}{\mathrm{gr}}   
\newcommand{\wh}{\mathrm{wh}}   
\newcommand{\hl}{\mathrm{hl}}   
\newcommand{\crs}{\mathrm{cr}}  

\newcommand{\G}[1]{\ensuremath{\mathrm{G#1}}}
\newcommand{\W}[1]{\ensuremath{\mathrm{W#1}}}

\begin{document}

\title[Lefschetz Fibrations on Knot Traces of Alternating Knots]{Lefschetz Fibrations on Knot Traces of Alternating and Extended Alternating Knots}
\author{Atsushi Tanaka}
\date{}

\begin{abstract}
In our previous work, we introduced a simple and explicit method for constructing a positive allowable Lefschetz fibration (PALF) from a $2$-handlebody decomposition of any given compact Stein surface. In this paper, we apply this construction to knot traces whose attaching circles are either alternating knots or \emph{extended alternating knots} (a generalized class introduced herein). We demonstrate that each such knot trace admits a PALF whose regular fiber has a genus exactly equal to the number of white regions in the associated planar graph, yielding PALFs whose regular fibers have a significantly small genus. As immediate corollaries, we prove that knot traces of positive pretzel knots with $s$ rows admit PALFs with regular fibers of genus $s-1$, and those of positive torus knots admit PALFs with regular fibers of genus $1$. Furthermore, we define \emph{positive torus-pretzel knots} by replacing each twist block of a positive pretzel knot with the crossings of a positive torus knot, and we establish that their knot traces also admit PALFs with regular fibers of genus $s-1$.
\end{abstract}

\maketitle

\section{Introduction}\label{sec:intro}

For standard definitions and foundational results concerning Lefschetz fibrations, we refer the reader to \cite{MR4327688}, \cite{MR1707327}, and \cite{MR2114165}. Loi and Piergallini \cite{MR1835390}, Akbulut and Ozbagci \cite{MR1825664}, and Akbulut and Arikan \cite{MR2972525} independently proved that every compact Stein surface (a Stein surface in short) admits a positive allowable Lefschetz fibration over the disk $D^2$ with bounded fibers (PALF in short), and they provided explicit constructions of such fibrations.

In our previous work \cite{Tanaka2025construction} and \cite{Tanaka2026combinatorial}, we presented a simple and systematic method for constructing a PALF from a $2$-handlebody decomposition of any given Stein surface. This method produces PALFs whose regular fibers have a significantly small genus, thereby providing an alternative constructive proof of the aforementioned result. 

Throughout this paper, we use the term ``knot'' to include links as well. Unless otherwise specified, we assume that all knots are \emph{connected, prime, and represented by reduced diagrams (i.e., diagrams without nugatory crossings)}. We apply our novel construction method to knot traces whose attaching circles are either alternating knots or \emph{extended alternating knots}---a generalized class of knots introduced in this paper, which we formally define in Section~\ref{sec:extended}.

The main results of this paper are formalized in the following two theorems.

\begin{thm}\label{thm:alternating_1}
Let $X$ be a Stein surface consisting of a single $0$-handle and $m$ $2$-handles ($m \geq 1$) attached along a Legendrian knot whose topological type is an alternating knot satisfying Condition~$\mathrm{M}$ (defined in Section~\ref{sec:preliminaries}). Consider the planar graph associated with this alternating knot. Then, there exists a positive allowable Lefschetz fibration (PALF) whose total space is diffeomorphic to $X$, and whose regular fiber has a genus exactly equal to the number of white regions in the corresponding planar graph.
\end{thm}

\begin{thm}\label{thm:extended_1}
Let $X$ be a Stein surface consisting of a single $0$-handle and $m$ $2$-handles ($m \geq 1$) attached along a Legendrian knot whose topological type is an extended alternating knot. Consider the planar graph associated with this extended alternating knot. Then, there exists a positive allowable Lefschetz fibration (PALF) whose total space is diffeomorphic to $X$, and whose regular fiber has a genus exactly equal to the number of white regions in the corresponding planar graph.
\end{thm}

As direct corollaries of these theorems, we show that knot traces whose attaching circles are positive pretzel knots with $s$ rows admit PALFs with regular fibers of genus $s-1$. Similarly, those whose attaching circles are positive torus knots admit PALFs with regular fibers of genus $1$.

Furthermore, we consider \emph{positive torus-pretzel knots}, which are obtained by replacing each twist block of a positive pretzel knot with a block consisting of the crossings of a positive torus knot. We prove that knot traces whose attaching circles are positive torus-pretzel knots with $s$ rows also admit PALFs with regular fibers of genus $s-1$.

The remainder of this paper is organized as follows. Section~\ref{sec:preliminaries} reviews the relevant definitions and theoretical background, including the planar graphs corresponding to alternating knots. Section~\ref{sec:alternating} applies our geometric construction to knot traces of alternating knots, demonstrating that any such knot trace admits a PALF whose regular fiber genus is determined by the number of white regions in the corresponding planar graph. As an immediate application, we prove the genus formula for positive pretzel knots. Section~\ref{sec:extended} formally introduces extended alternating knots and establishes that any such knot trace similarly admits a PALF whose regular fiber genus is determined by the number of white regions in the corresponding planar graph. We also define positive torus-pretzel knots and establish the genus formula for positive torus-pretzel knots in this section. Finally, Section~\ref{sec:non_maximal_tb} details an additional modification method required for the specific scenario where the Thurston--Bennequin (tb) number of the Legendrian knot $\widetilde{C_{0k}}$ ($1 \leq k \leq m$) is strictly less than its maximal value.

\section*{Acknowledgments}
The author would like to express his sincere gratitude to Professor Hisaaki Endo for his invaluable guidance and continuous support. His emphasis on the concrete construction and study of $4$-manifolds profoundly deepened the author's appreciation of this rich field. The author also thanks his colleagues at the Institute of Science Tokyo for their helpful discussions and insightful comments. This work was supported by JST SPRING, Japan Grant Number JPMJSP2180.

\section{Preliminaries}\label{sec:preliminaries}

Throughout this paper, we use the term ``knot'' to include links as well. Unless otherwise specified, we assume that all knots are \emph{connected, prime, and represented by reduced diagrams (i.e., diagrams without nugatory crossings)}. In this section, we describe the properties of Legendrian knots when converted into grid position, and we formally define the classes of knots studied in this work.

In \cite{Tanaka2025construction} and \cite{Tanaka2026combinatorial}, we introduced a method for constructing a PALF $P$ that is diffeomorphic to a given Stein surface $\Pi$, based on its $2$-handlebody decomposition. Here, we focus on the case where the initial Stein surface $\Pi$ is a $2$-handlebody consisting of a single $0$-handle and $m$ $2$-handles ($m \geq 1$) attached along a Legendrian knot; namely, the case of a knot trace.

As detailed in Subsection 2.2 of \cite{Tanaka2025construction}, the procedure for converting any Legendrian knot in $(S^3, \xi_{st})$ into a knot in grid position proceeds as follows: first, we rotate the Legendrian knot $\widetilde{C_{0k}}$ ($1 \leq k \leq m$) clockwise by $45^\circ$. Next, we deform it into a knot $\widetilde{C_{0k}'}$ in grid position such that the left cusps map to northwest (NW) corners and the right cusps map to southeast (SE) corners. Consequently, the number of NW corners of the knot $\widetilde{C_{0k}'}$ in grid position corresponds exactly to the number of left cusps of the original Legendrian knot $\widetilde{C_{0k}}$.

The subsequent discussion is based on \cite{MR2500576} and \cite[Chapter 12]{MR3381987}.

\begin{rem}\label{rem:Legendre}
In this paper, we adopt the convention of \cite{MR2500576}. Readers consulting \cite{MR3381987} should note the difference highlighted in \cite{MR2500576}: ``Note that our convention differs from the convention of \cite{MR3381987}: the convention there is to reverse all crossings in the grid diagram and then rotate $45^\circ$ clockwise.''
\end{rem}

The following moves on a grid diagram preserve the Legendrian isotopy class of the corresponding Legendrian knot:

\begin{itemize}
\item \textbf{Translation.}
By viewing the grid diagram as being embedded on a torus---obtained by identifying the opposite sides of the grid---one can apply cyclic translations to the diagram. Any such translation can be expressed as a composition of elementary vertical and horizontal shifts. Specifically, a vertical translation cyclically permutes the rows (e.g., moving the top row to the bottom), while a horizontal translation cyclically permutes the columns (e.g., moving the leftmost column to the rightmost).

\item \textbf{Commutation.}
A commutation interchanges two adjacent rows (vertical commutation) or two adjacent columns (horizontal commutation), provided that the segments within the adjacent rows or columns are either strictly disjoint or nested.

\item \textbf{(De)stabilization.}
Stabilizations of types $X$:NE, $X$:SW, $O$:NE, and $O$:SW, along with their inverse destabilizations, preserve the Legendrian isotopy class of the knot.
\end{itemize}

As in \cite{Tanaka2025construction}, the orientation of the knots is irrelevant for our current purposes. Consequently, we omit the standard $X$ and $O$ markings from the grid diagrams, depicting only the vertical and horizontal segments (see Figure~\ref{fig:Y0201}). In particular, when performing an NE stabilization at an SW corner, or an SW stabilization at an NE corner, the writhe of the knot increases by one. However, since the number of NW corners (which correspond to left cusps) or SE corners (which correspond to right cusps) also increases by one, the Thurston--Bennequin number remains invariant.

In contrast, NW and SE stabilizations decrease the Thurston--Bennequin number by one. Therefore, applying these stabilizations yields a knot that is not Legendrian isotopic to the original one (see Figure~\ref{fig:Y0202}).

\begin{figure}[htbp]
\centering  
\begin{tikzpicture}
    \node[anchor=south west, inner sep=0] (image) at (0,0)  {\includegraphics[scale=0.6]{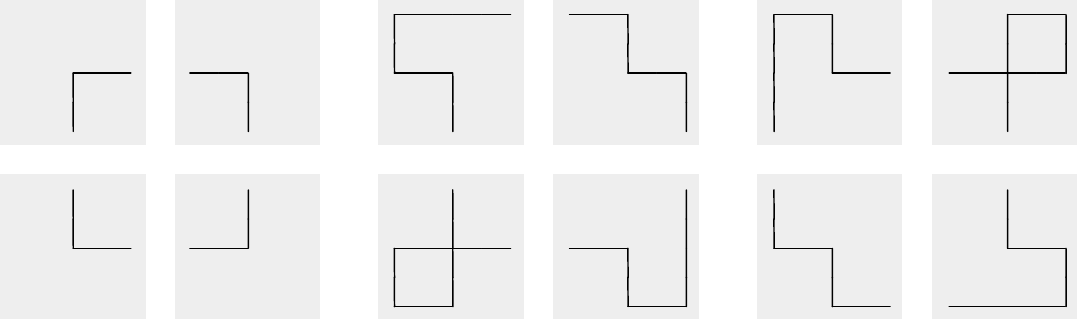}};
    \begin{scope}[x={(image.south east)},y={(image.north west)}]
        \node [below] at (0.15, -0.04) {(a)};
        \node [below] at (0.5, -0.04) {(b)};
        \node [below] at (0.85, -0.04) {(c)};
    \end{scope}
\end{tikzpicture}
\caption{Stabilizations preserving the Legendrian isotopy class: (a) original state; (b) after NE stabilization; (c) after SW stabilization.}
\label{fig:Y0201}
\end{figure}
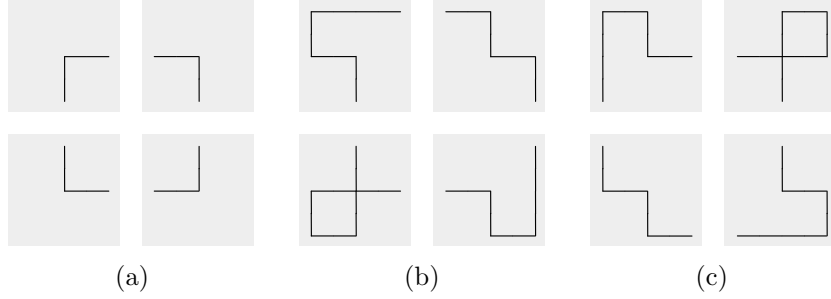

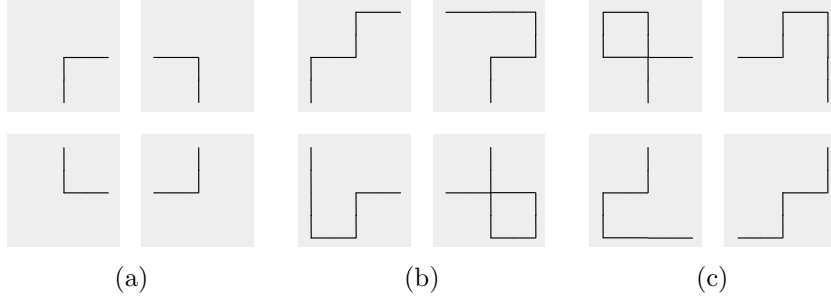
\begin{figure}[htbp]
\centering  
\begin{tikzpicture}
    \node[anchor=south west, inner sep=0] (image) at (0,0)  {\includegraphics[scale=0.6]{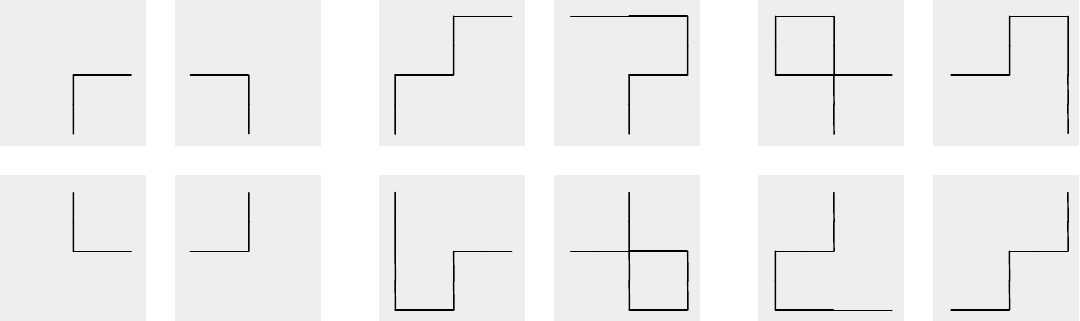}};
    \begin{scope}[x={(image.south east)},y={(image.north west)}]
        \node [below] at (0.15, -0.04) {(a)};
        \node [below] at (0.5, -0.04) {(b)};
        \node [below] at (0.85, -0.04) {(c)};
    \end{scope}
\end{tikzpicture}
\caption{Stabilizations changing the Legendrian isotopy class: (a) original state; (b) after NW stabilization; (c) after SE stabilization.}
\label{fig:Y0202}
\end{figure}

In Section~\ref{sec:alternating}, we consider the case where the attaching circles of the $2$-handles form an alternating knot. The correspondence between knots and planar graphs described below follows the approach of \cite{MR2079925}.

Suppose a knot in grid position is given on the plane. Assigning white to the unbounded outer region, we color the bounded complementary regions of the knot such that the entire plane, including the outer region, alternates between gray and white. (Note that the unbounded outer region is not included in the total count of white regions.)

We then place a vertex at the center of each gray region and connect two vertices with an edge if their corresponding regions share a crossing. The resulting planar graph is thus uniquely associated with the given knot in grid position. Consequently, there is a natural one-to-one correspondence between the following: the gray regions and the vertices of the graph, the white regions and the faces bounded by the edges, and the crossings of the knot and the edges of the graph.

In particular, when the knot is alternating, only the following two consistent coloring patterns can occur at the crossings:

\begin{itemize}
\item[\textbf{(a)}] In the neighborhood of every crossing, the second (upper-left) and fourth (lower-right) quadrants are shaded gray (see Figure~\ref{fig:Y0203}(a)).
\item[\textbf{(b)}] In the neighborhood of every crossing, the first (upper-right) and third (lower-left) quadrants are shaded gray (see Figure~\ref{fig:Y0203}(b)).
\end{itemize}

Figures~\ref{fig:Y0204}(a) and (b) illustrate examples of the right-handed trefoil knot in grid position corresponding to types (a) and (b), respectively.

\begin{figure}[htbp]
\centering  
\begin{tikzpicture}
    \node[anchor=south west, inner sep=0] (image) at (0,0)  {\includegraphics[scale=0.8]{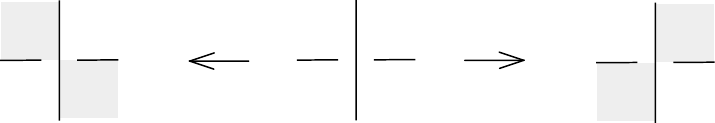} };
    \begin{scope}[x={(image.south east)},y={(image.north west)}]
        \node [below] at (0.09, -0.02) {(a)};
        \node [below] at (0.91, -0.02) {(b)};
    \end{scope}
\end{tikzpicture}
\caption{Checkerboard coloring of the knot diagram.}
\label{fig:Y0203}
\end{figure}

\begin{figure}[htbp]
\centering  
\begin{tikzpicture}
    \node[anchor=south west, inner sep=0] (image) at (0,0)  {\includegraphics[scale=0.6]{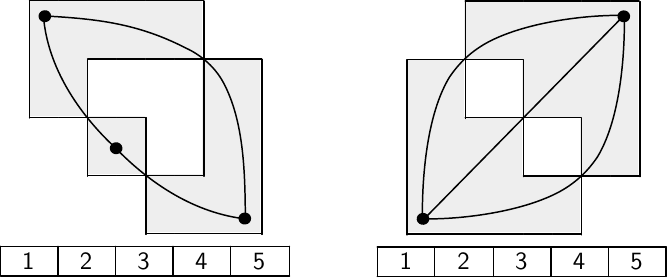}};
    \begin{scope}[x={(image.south east)},y={(image.north west)}]
        \node [below] at (0.215, -0.02) {(a)};
        \node [below] at (0.785, -0.02) {(b)};
    \end{scope}
\end{tikzpicture}
\caption{Diagrams of the right-handed trefoil knot of types (a) and (b).}
\label{fig:Y0204}
\end{figure}

In this paper, we restrict our attention to alternating knots in grid position of type (a). Furthermore, to ensure a simplified construction, we only consider knots that satisfy the following condition, referred to as Condition~$\mathrm{M}$. (The formal definition of a Mondrian diagram and the specific geometric rationale for taking the mirror image will be detailed later.)

\begin{itemize}
\item \textbf{Condition $\mathrm{M}$:} There exists an alternating knot in grid position derived from a Mondrian diagram that possesses exactly one local maximum segment and exactly one local minimum segment, such that its mirror image coincides exactly with the given knot in grid position.
\end{itemize}

Given a connected, prime, and reduced alternating knot diagram, one can construct---via a Mondrian diagram---a knot in grid position that satisfies Condition~$\mathrm{M}$ and realizes the maximal Thurston--Bennequin number for that knot type.

Suppose we are given a knot $\widetilde{C_{0k}'}$ ($1 \leq k \leq m$) in grid position, obtained by converting a Legendrian knot $\widetilde{C_{0k}}$ in $(S^3, \xi_{st})$. In Sections~\ref{sec:alternating} and \ref{sec:extended}, provided that $\widetilde{C_{0k}'}$ satisfies Condition~$\mathrm{M}$, we restrict our focus to the case where the Thurston--Bennequin (tb) number of $\widetilde{C_{0k}'}$ (or equivalently, that of the original $\widetilde{C_{0k}}$) is maximal. Under this assumption, the construction detailed below can be applied without modification. The scenario where the tb number is strictly less than the maximal value will be addressed in Section~\ref{sec:non_maximal_tb}. In that section, we will introduce an additional operation to compensate for the deficit from the maximal tb number, thereby adapting the current argument to the general case.

The following procedure is based on the method presented in \cite[Section 3]{MR2186113}, modified to suit the specific purposes of this paper. The original procedure in \cite[Section 3]{MR2186113} always yields an alternating knot in grid position of type (b). To obtain a knot of type (a) instead, we introduce Steps (1) and (6) to the algorithm: we first take the mirror image of the initial knot diagram, apply the core transformations, and finally take the mirror image of the resulting grid diagram. 

As a concrete example, Figure~\ref{fig:Y0205} illustrates the sequence of steps to construct a type (a) grid position for the right-handed trefoil knot. The subfigures (1) through (6) display the results of applying operations (1) to (6) sequentially to the initial knot diagram (0).

\begin{itemize} 
\item[\textbf{(1)}] Construct a knot diagram for the mirror image $\bar{K}$ of a given alternating knot $K$.
\item[\textbf{(2)}] Construct the associated planar graph from the alternating knot diagram of $\bar{K}$. Color the complementary regions alternately gray and white such that the coloring pattern around each crossing satisfies the condition of type (b), as depicted in Figure~\ref{fig:Y0203}(b).
\item[\textbf{(3)}] Transform this planar graph into a Mondrian diagram, where the vertices correspond to horizontal segments and the edges correspond to vertical segments, ensuring that the resulting diagram possesses exactly one local maximum segment and exactly one local minimum segment.
\end{itemize} 

\begin{dfn} \label{dfn:Mondrian}
A \emph{Mondrian diagram} is a planar diagram consisting of mutually disjoint horizontal segments and vertical segments. The endpoints of each vertical segment must lie on distinct horizontal segments, and the interior of any vertical segment must not intersect any other horizontal or vertical segments in the plane. In a Mondrian diagram, a \emph{local maximum segment} is a horizontal segment that has no vertical segments attached from above. Similarly, a \emph{local minimum segment} is a horizontal segment that has no vertical segments attached from below.
\end{dfn} 

\begin{itemize} 
\item[\textbf{(4)}] We can now convert this Mondrian diagram into a front projection as follows: replace each horizontal segment with a ``pair of lips'' front for the unknot, delete from these fronts a small neighborhood around each intersection with a vertical segment, and then replace each vertical segment with a crossing. This front projection represents a Legendrian knot corresponding to the given alternating knot diagram. 
\item[\textbf{(5)}] Apply the previously described conversion procedure to transform the front projection of the Legendrian knot into a knot $\bar{K}$ in grid position. In the resulting grid diagram, 
\item[\textbf{(6)}] Rotate the entire grid diagram counterclockwise by $90^\circ$. This geometric operation corresponds to taking the mirror image, thereby yielding the desired knot $K$ in grid position of type (a). Consequently, the neighborhoods of the upper-left (NW) and lower-right (SE) corners of the grid diagram are now shaded gray. Finally, we refer to the leftmost vertex of the associated planar graph as the \emph{starting point} and the rightmost vertex as the \emph{terminal point}.

Note the following crucial property for performing our simplified construction: when the edges of the planar graph are oriented in the direction from the starting point toward the terminal point, any edge exiting a gray region strictly proceeds from the second (upper-left) quadrant to the fourth (lower-right) quadrant. (An edge never travels backward from the fourth quadrant to the second quadrant.)
\end{itemize}

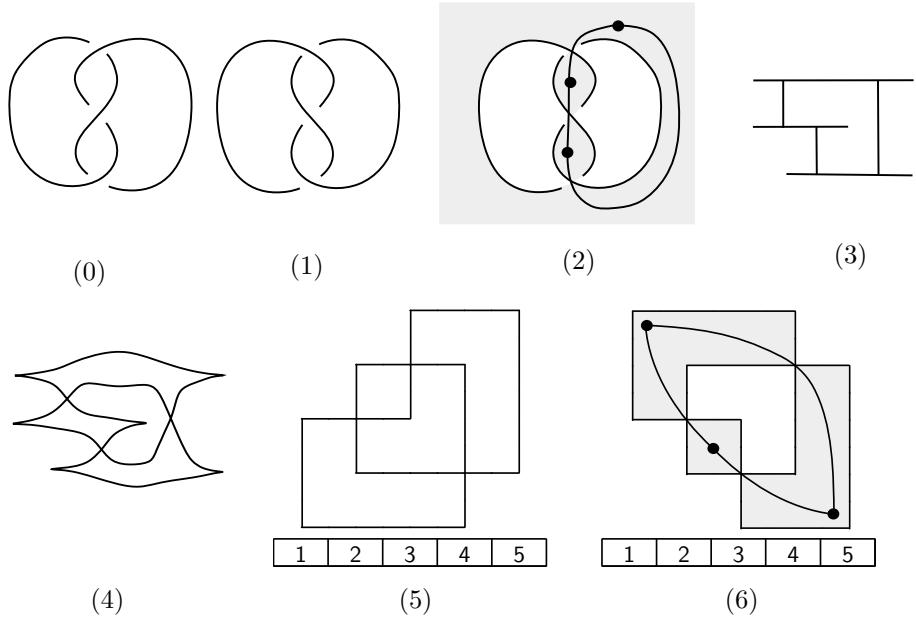
\begin{figure}[htbp]
\centering  
\begin{tikzpicture}
    \node[anchor=south west, inner sep=0] (image) at (0,0)  {\includegraphics[scale=0.73]{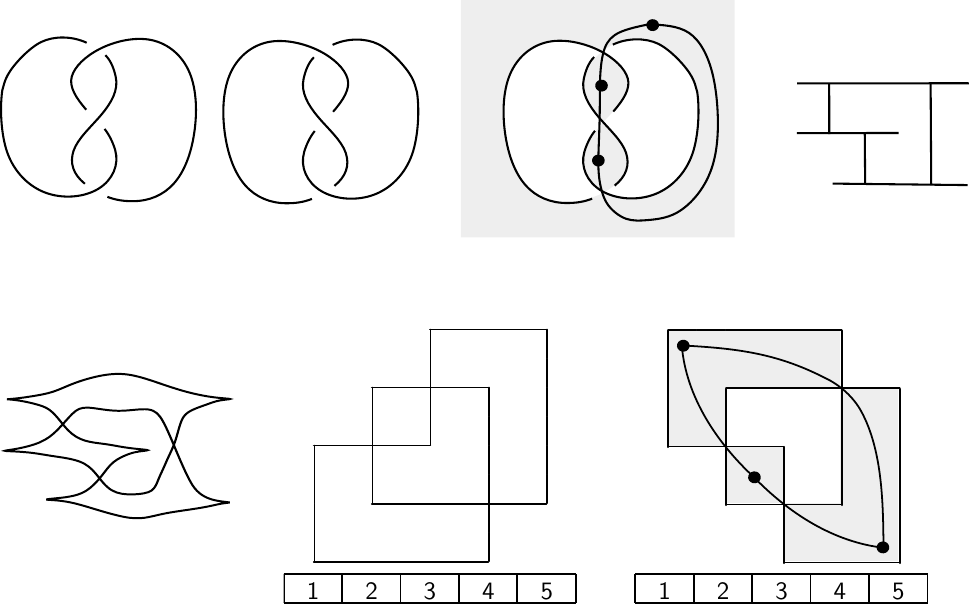} };
    \begin{scope}[x={(image.south east)},y={(image.north west)}]
        \node [below] at (0.09, 0.56) {(0)};
        \node [below] at (0.33, 0.57) {(1)};
        \node [below] at (0.63, 0.58) {(2)};
        \node [below] at (0.93, 0.59) {(3)};
        \node [below] at (0.11, -0.02) {(4)};
        \node [below] at (0.45, -0.02) {(5)};
        \node [below] at (0.81, -0.02) {(6)};
    \end{scope}
\end{tikzpicture}
\caption{The procedure for constructing an alternating knot in grid position that satisfies Condition~$\mathrm{M}$.}
\label{fig:Y0205}
\end{figure}

Furthermore, in Subsection~\ref{subsec:alternating_application}, we specifically examine positive pretzel knots.

\begin{dfn}\label{def:pretzel}
A \emph{positive pretzel knot}, denoted $P(q_1, q_2, \ldots, q_s)$, is a knot of the form illustrated in Figure~\ref{fig:Y0206}. Here, the number of rows $s$ satisfies $s \geq 3$, and each row contains $q_k$ right-handed half-twists, with $q_k \geq 1$ for all $1 \leq k \leq s$.
\end{dfn}

\begin{figure}[htbp]
\centering
\includegraphics[scale=0.8]{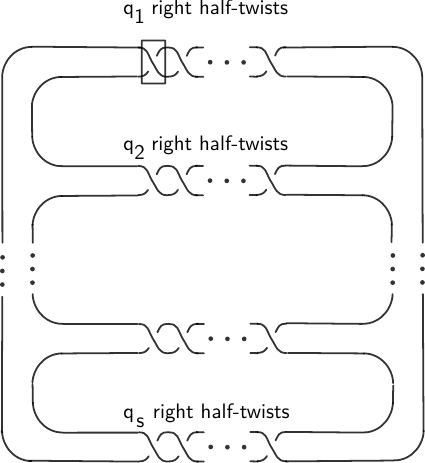}
\caption{A diagram of the positive pretzel knot $P(q_1, q_2, \dots, q_s)$.}
\label{fig:Y0206}
\end{figure}

Subsequently, in Subsection~\ref{subsec:extended_application}, we extend our analysis to positive torus knots.

\begin{dfn}\label{def:torus}
A \emph{positive torus knot}, denoted $T_{p,q}$, is a knot of the form depicted in Figure~\ref{fig:Y0207}. It is defined such that the number of strands $p$ satisfies $p > 1$, and the number of crossing boxes $q$ satisfies $q > 1$. Note that, in general, $T_{p,q}$ consists of $\gcd(p, q)$ distinct components.
\end{dfn}

\begin{figure}[htbp]
\centering
\includegraphics[scale=0.65]{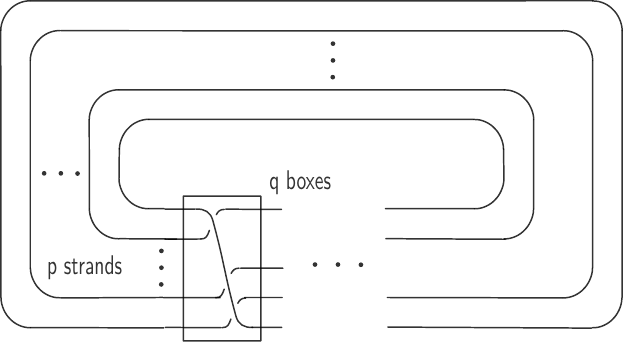}
\caption{A diagram of the positive torus knot $T_{p,q}$.}
\label{fig:Y0207}
\end{figure}

\section{Lefschetz fibrations on knot traces of alternating knots}\label{sec:alternating}

\subsection{PALF construction and the main theorem}\label{subsec:alternating_main}

We now state the main theorem of this section together with its underlying assumptions. Let the initial Stein surface $\Pi$ be a $2$-handlebody consisting of a single $0$-handle and $m$ $2$-handles ($m \geq 1$) attached along a Legendrian knot whose topological type is an alternating knot. Following the procedure outlined in Section~\ref{sec:preliminaries}, we convert each Legendrian knot $\widetilde{C_{0k}}$ ($1 \leq k \leq m$) in $(S^3, \xi_{st})$ into a corresponding knot $\widetilde{C_{0k}'}$ in grid position. For the subsequent discussion, we restrict our attention to the case where Condition~$\mathrm{M}$ (introduced in Section~\ref{sec:preliminaries}) is satisfied. Under these premises, we obtain the following theorem.

\begin{thm}\label{thm:alternating}
Let $X$ be a Stein surface consisting of a single $0$-handle and $m$ $2$-handles ($m \geq 1$) attached along a Legendrian knot whose topological type is an alternating knot satisfying Condition~$\mathrm{M}$. 
Consider the planar graph associated with this alternating knot. Then, there exists a positive allowable Lefschetz fibration (PALF) whose total space is diffeomorphic to $X$, and whose regular fiber has a genus exactly equal to the number of white regions in the corresponding planar graph.
\end{thm}

In this section, each vertical segment of the knot $\widetilde{C_{0k}'}$ ($1 \leq k \leq m$) in grid position that crosses a horizontal segment is deformed into a specific arc. This resulting arc contains a vertical segment of exactly two lattice units in height, bounded by an NE corner and an SW corner (see Figure~\ref{fig:Y0301}). We denote the knot in grid position obtained after this deformation by $\widetilde{C_{0k}''}$ ($1 \leq k \leq m$). This deformation is achieved through a sequence of stabilizations and commutations that preserve the Legendrian isotopy class of the knot, as detailed in Section~\ref{sec:preliminaries}:

\begin{itemize}
\item If the upper endpoint of a vertical segment is an NW corner (respectively, an NE corner), we perform an NE stabilization at that corner (if necessary) and subsequently move the horizontal segment downward via commutations (Figure~\ref{fig:Y0301} (i)--(iii)).
\item If the lower endpoint of a vertical segment is an SE corner (respectively, an SW corner), we perform an SW stabilization at that corner (if necessary) and subsequently move the horizontal segment upward via commutations (Figure~\ref{fig:Y0301} (iv)--(v)).
\end{itemize}

\begin{figure}[htbp]
\centering  
\begin{tikzpicture}
    \node[anchor=south west, inner sep=0] (image) at (0,0)  {\includegraphics[scale=0.6]{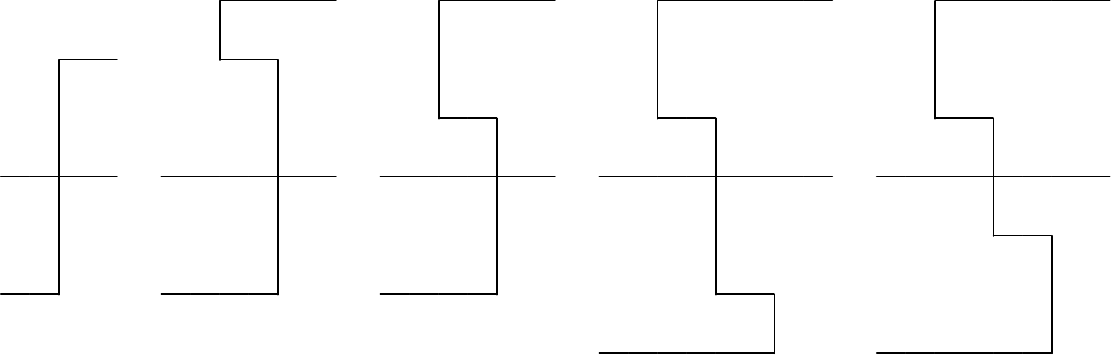}};
    \begin{scope}[x={(image.south east)},y={(image.north west)}]
        \node [below] at (0.05, -0.04) {(i)};
        \node [below] at (0.22, -0.04) {(ii)};
        \node [below] at (0.43, -0.04) {(iii)};
        \node [below] at (0.65, -0.04) {(iv)};
        \node [below] at (0.89, -0.04) {(v)};
    \end{scope}
\end{tikzpicture}
\caption{Deformation of a vertical segment intersecting a horizontal segment.}
\label{fig:Y0301}
\end{figure}
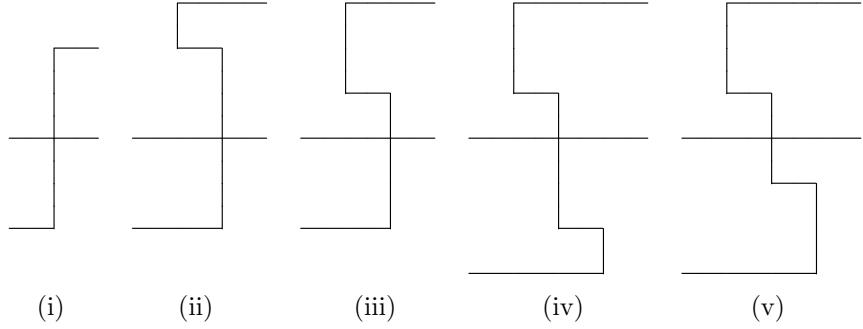

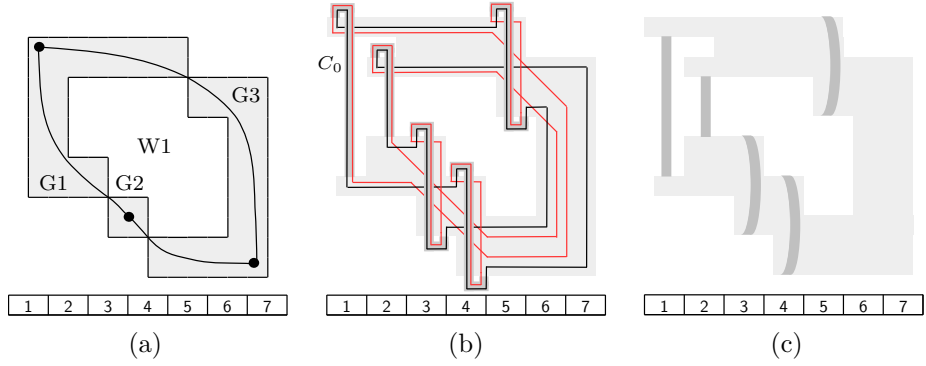
\begin{figure}[htbp]
\centering  
\begin{tikzpicture}
    \node[anchor=south west, inner sep=0] (image) at (0,0)  {\includegraphics[scale=0.65]{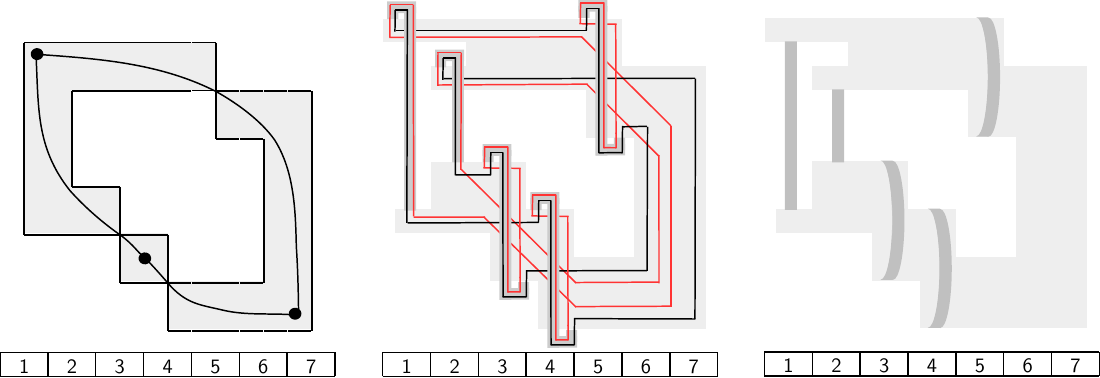} };
    \begin{scope}[x={(image.south east)},y={(image.north west)}]
        \node [below] at (0.15, -0.02) {(a)};
        \node [below] at (0.5, -0.02) {(b)};
        \node [below] at (0.85, -0.02) {(c)};
        
        \node[font=\scriptsize] at (0.352, 0.8) {$C_0$};
        \node at (0.05, 0.425) {\footnotesize{G1}};
        \node at (0.133, 0.425) {\footnotesize{G2}};
        \node at (0.26, 0.703) {\footnotesize{G3}};
        \node at (0.16, 0.54) {\footnotesize{W1}};
    \end{scope}
\end{tikzpicture}
\caption{The trefoil knot example: (b) Monodromy factorization: $(C_0, C_5, C_4, C_3, C_2, C_1)$. Each $C_i$ ($1 \leq i \leq 5$) denotes a red simple closed curve passing over the $1$-handle in the $i$-th column.}
\label{fig:Y0302}
\end{figure}

As a first concrete example, applying the aforementioned deformation to the diagram of the trefoil knot shown in Figure~\ref{fig:Y0204}(a) yields the diagram in Figure~\ref{fig:Y0302}(a).

Regarding the PALF construction, this paper extends the method introduced in \cite{Tanaka2026combinatorial} as follows (see Figure~\ref{fig:Y0302}(b)):
In \cite{Tanaka2026combinatorial}, the procedure applied to a vertical segment with an NE corner was illustrated in Section~2, Figure~5(b) of that paper, with the corresponding Kirby move shown in Section~2, Figure~7. This specific Kirby move leaves the framing of the vertical segment with the NE corner unchanged. Furthermore, the attaching circle corresponding to the newly added simple closed curve has a writhe of $-1$ and a framing of $-2$.

In the present work, we apply a modified procedure to vertical segments with NE corners, as illustrated in Figure~\ref{fig:Y0305}(a). The corresponding Kirby move, depicted in Figure~\ref{fig:Y0305}(b), is a modification of the one shown in Section~2, Figure~7 of \cite{Tanaka2026combinatorial}. Crucially, the framing of the vertical segment with the NE corner remains unchanged, and the attaching circle corresponding to the newly added simple closed curve again has a writhe of $-1$ and a framing of $-2$. Consequently, just as in \cite{Tanaka2026combinatorial}, the total space of the PALF constructed via this updated procedure remains diffeomorphic to the original Stein surface.

Finally, Figure~\ref{fig:Y0302}(c) illustrates the regular fiber, which has been deformed to clearly display its boundary components. This regular fiber is a surface of genus $1$ with $4$ boundary components.

\begin{figure}[htbp]
\centering  
\begin{tikzpicture}
    \node[anchor=south west, inner sep=0] (image) at (0,0)  {\includegraphics[scale=0.7]{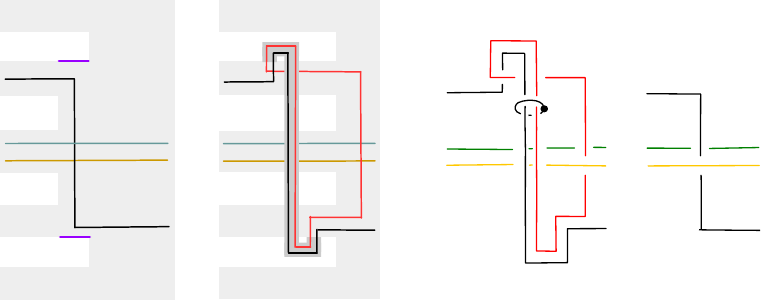}};
    \begin{scope}[x={(image.south east)},y={(image.north west)}]
        \node [below] at (0.255, -0.04) {(a)};
        \node [below] at (0.815, -0.04) {(b)};
    \end{scope}
\end{tikzpicture}
\caption{Handling a vertical segment with an NE corner in the PALF construction: (a) the new procedure; (b) the corresponding Kirby move.}
\label{fig:Y0305}
\end{figure}

As a second concrete example, the diagram obtained by applying the aforementioned deformation is shown in Figure~\ref{fig:Y0303}. The knot $\widetilde{C_{0}''}$ in grid position has a writhe of $-3$, five left cusps, a Thurston--Bennequin number of $-8$, and a framing of $-9$. Its topological type corresponds to the knot $7_6$ in the Rolfsen knot table.

Following the coloring rules described in Section~\ref{sec:preliminaries}, the appropriate regions are shaded gray, and the associated planar graph is superimposed. There are four gray regions, denoted $\G{1}$, $\G{2}$, $\G{3}$, and $\G{4}$, and four white regions, denoted $\W{1}$, $\W{2}$, $\W{3}$, and $\W{4}$.

The guide line $B_0$ is a copy of $\widetilde{C_0''}$ drawn on the $0$-handle $D^2$, depicted as a gray square (see Figure~\ref{fig:Y0304}). We then apply the PALF construction procedure detailed below. Although the full details will be explained later, the regular fiber and vanishing cycles of the resulting PALF $P$ are illustrated in Figure~\ref{fig:Y0306}.

Throughout this paper, we let $C_0$ (or $C_{0k}$) denote the simple closed curves corresponding to the vanishing cycles after the construction. Furthermore, we let $C_i$ ($i \geq 1$) denote the red simple closed curves (vanishing cycles) that pass over the $1$-handles added in the $i$-th column during the construction process.

\begin{figure}[htbp]
  \begin{minipage}[t]{0.49\textwidth}
    \centering
         \begin{tikzpicture}
            \node[anchor=south west, inner sep=0] (image) at (0,0)    {\includegraphics[scale=0.32]{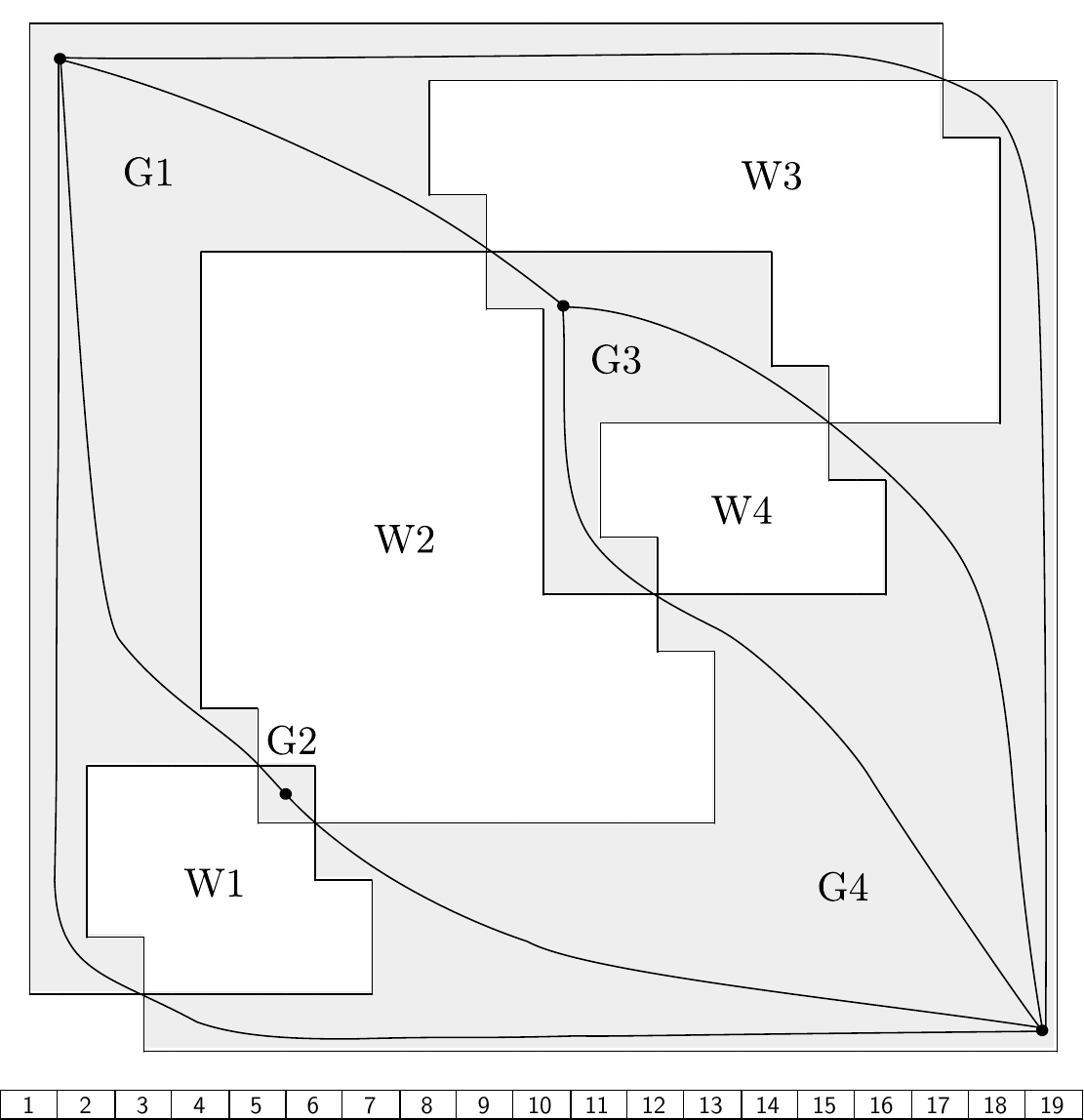}};
            \begin{scope}[x={(image.south east)},y={(image.north west)}]
            \node[font=\scriptsize] at (0, 0.87) {$\widetilde{C_0''}$};
            \end{scope}
         \end{tikzpicture}
    \caption{An alternating knot diagram and its associated planar graph.}
    \label{fig:Y0303}
  \end{minipage}
  \hfill
  \begin{minipage}[t]{0.49\textwidth}
    \centering
       \begin{tikzpicture}
            \node[anchor=south west, inner sep=0] (image) at (0,0)    {\includegraphics[scale=0.32]{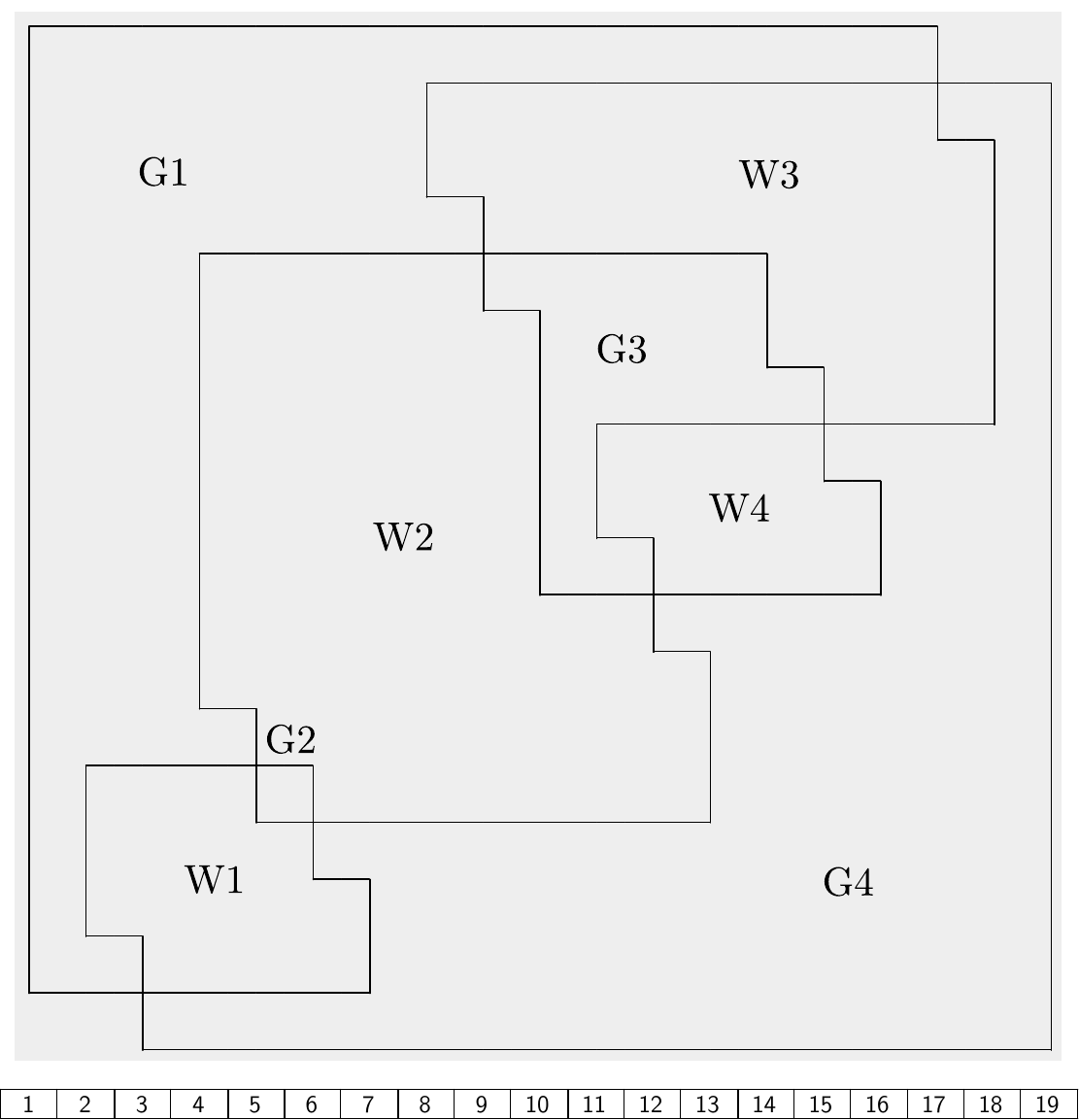}};
            \begin{scope}[x={(image.south east)},y={(image.north west)}]
            \node[font=\scriptsize] at (0, 0.87) {$B_0$};
            \end{scope}
         \end{tikzpicture}
    \caption{The closed curve $B_0$ on the surface of the $0$-handle.}
    \label{fig:Y0304}
  \end{minipage}
\end{figure}

Our analysis proceeds in three steps: first, we construct the PALF by classifying the vertical segments and their corresponding $1$-handles; next, we analyze the boundary components of the regular fiber; and finally, we compute its genus.

The vertical segments of $B_0$, along with the $1$-handles attached to the boundary of the $0$-handle to lift them, can be classified into four distinct types:

\begin{description}
\item[Case (1)]  \textbf{Vertical segments with NW corners located at the left edges of white regions.}  
Each white region contains exactly one vertical segment with an NW corner at its left edge. 
In our running example, these segments appear at the left edges of the four white regions $\W{1}$, $\W{2}$, $\W{3}$, and $\W{4}$, specifically located in columns 2, 4, 8, and 11. Each of these segments is lifted by attaching a $1$-handle to the boundary of the $0$-handle. Subsequently, an isotopy is applied to deform the $0$-handle, effectively pushing it out of the white regions except in the immediate neighborhoods of the simple closed curves.

\item[Case (2)] \textbf{Vertical segments with NW corners other than those in case (1).}  
In our running example, one such segment appears at the left edge of $\G{1}$ in column~1. After attaching a $1$-handle in column~1, we apply an isotopy to push the boundary of the $0$-handle to the right, thereby forming a hole strictly inside the region $\G{1}$. 
\end{description}

For all vertical segments featuring an NW corner (i.e., those categorized under cases (1) and (2)), we apply the procedure illustrated in Section~2, Figure~5(a) of \cite{Tanaka2026combinatorial}. It is worth noting that none of the vertical segments with NW corners intersect any horizontal segments.

\begin{description}
\item[Case (3)] \textbf{Vertical segments with NE corners that cross horizontal segments.}  
These segments are precisely the vertical segments with NE corners generated by the deformation described earlier (see Figure~\ref{fig:Y0301}), wherein each vertical segment of the knot $\widetilde{C_{0k}'}$ ($1 \leq k \leq m$) in grid position that crosses a horizontal segment is deformed into a specific arc. Geometrically, each such vertical segment lies along the boundary separating two distinct gray regions. In our specific example, they appear in columns~3, 5, 6, 9, 12, 15, and~17. For instance, the vertical segment in column~3 separates $\G{1}$ and $\G{4}$. 

For every vertical segment of this type, the region immediately above its upper endpoint and the region immediately below its lower endpoint each correspond to either a white region or the unbounded outer region. Returning to the vertical segment in column~3, the region above its upper endpoint is $\W{1}$, while the region below its lower endpoint is the unbounded outer region.

After each of these vertical segments is lifted by attaching a $1$-handle, we push the boundary of the $0$-handle to the right---passing it underneath the newly attached $1$-handle---only if necessary. To illustrate, at the vertical segment in column~3, the boundary of the $0$-handle is \emph{not} pushed underneath the $1$-handle. Conversely, at the vertical segment in column~9, the boundary is indeed pushed to the right, passing underneath the $1$-handle (the precise conditions for executing this operation will be detailed later).

\item[Case (4)] \textbf{Vertical segments with NE corners that do not cross horizontal segments.}  
No $1$-handles are attached to these segments.
\end{description}

\begin{figure}[htbp]
\centering  
\begin{tikzpicture}
    \node[anchor=south west, inner sep=0] (image) at (0,0)  {\includegraphics[scale=0.60]{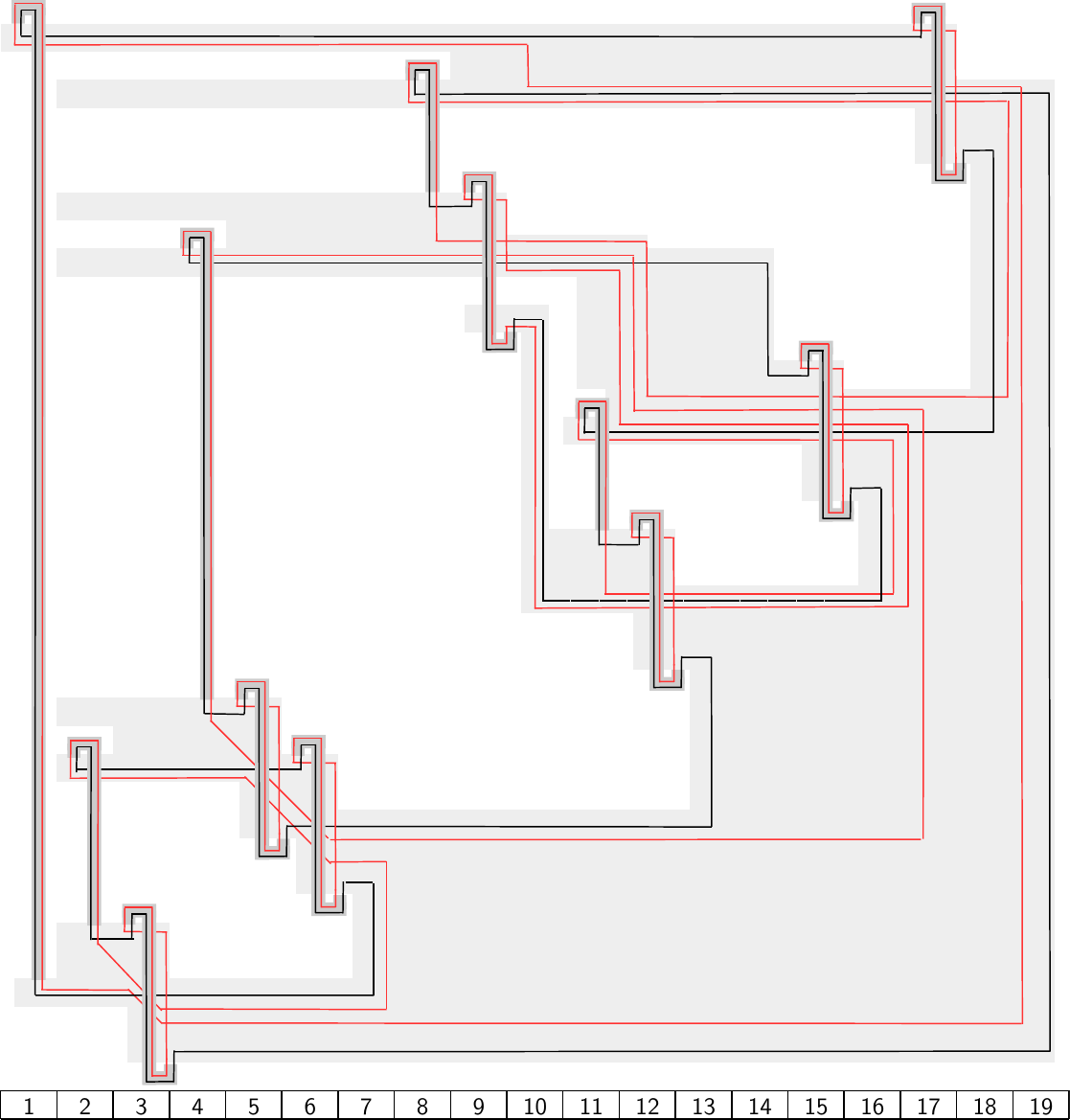}};
    \begin{scope}[x={(image.south east)},y={(image.north west)}]
        \node at (0.005, 0.86) {$C_0$};
        \node at (0.14, 0.86) {G1};
        \node at (0.29, 0.36) {G2};
        \node at (0.66, 0.7) {G3};
        \node at (0.8, 0.16) {G4};
        \node at (0.18, 0.22) {W1};
        \node at (0.35, 0.53) {W2};
        \node at (0.78, 0.82) {W3};
        \node at (0.7, 0.53) {W4};
    \end{scope}
\end{tikzpicture}
\caption{Monodromy factorization:
$(C_0, C_{17}, C_{15}, C_{12}, C_{11}, C_9, C_8, C_6,  C_5, $
$C_4, C_3, C_2, C_1)$.
Each $C_i$ ($1 \leq i \leq 17$) denotes a red simple closed curve passing over the $1$-handle in the $i$-th column.}
\label{fig:Y0306}
\end{figure}
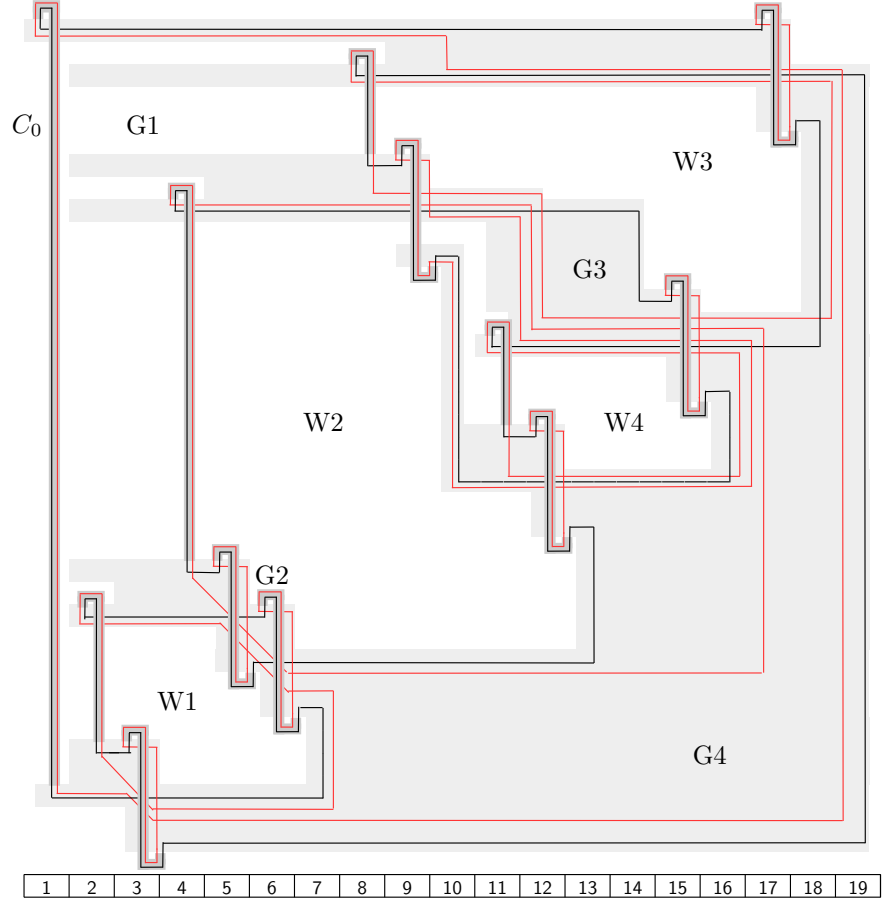

In the PALF construction process, it is a crucial requirement that the neighborhood immediately to the left of each vertical segment with an NW corner becomes part of the boundary of the $0$-handle. Returning to our specific example, at the left edge of the gray region $\G{1}$ in column~1, we attach a $1$-handle to the $0$-handle. This operation allows us to lift the vertical segment of the closed curve $B_0$ located in that column, after which we apply an isotopy to push the boundary of the $0$-handle to the right.

As a direct result of this isotopy, the neighborhoods immediately to the left of the vertical segments situated at the left edges of the white regions $\W{1}$, $\W{2}$, and $\W{3}$ (specifically in columns~2, 4, and~8, respectively) properly become part of the boundary of the $0$-handle. For each of these locations, we attach a $1$-handle to the boundary of the $0$-handle and lift the vertical segment in the corresponding column. We then deform the $0$-handle, effectively pushing it out of these white regions, except in the immediate vicinity of the simple closed curves.

Now, consider the vertical segments with NE corners that cross horizontal segments. If both the upper and lower endpoints of such a segment lie within $\W{1}$, $\W{2}$, $\W{3}$, or the unbounded outer region---as is the case for the vertical segments in columns~3, 5, 6, 9, and~17---we can readily attach $1$-handles to the adjacent boundary of the $0$-handle to lift them.

Next, we must ensure that the neighborhood to the left of the vertical segment in column~11, located at the left edge of $\W{4}$, also becomes part of the boundary of the $0$-handle.  Because the vertical segment in column~9 has already been lifted via a $1$-handle, we achieve this by first pushing the boundary of the $0$-handle to the right---sliding it underneath this pre-existing $1$-handle---and subsequently pushing it downward along columns~10 and~11.

Following the attachment of a $1$-handle and the lifting of the vertical segment in column~11, we deform the $0$-handle, effectively pushing it out of $\W{4}$ except in the immediate vicinity of the simple closed curves.

The remaining vertical segments with NE corners that cross horizontal segments---specifically those in columns~12 and~15---can similarly be lifted by attaching $1$-handles to the boundary of the $0$-handle. In practice, because the geometric construction proceeds sequentially from column~1 to column~17, the $1$-handles are attached in strictly increasing order of their column numbers; namely, columns~1, 2, 3, 4, 5, 6, 8, 9, 11, 12, 15, and~17.

Each red simple closed curve, introduced in association with an attached $1$-handle, is appropriately positioned on the $0$-handle. When the boundary of the $0$-handle is pushed to the right---sliding underneath the corresponding $1$-handle---the red curve is explicitly arranged to encircle the extruded portion of the $0$-handle boundary. 

Subsequently, by removing the guide line $B_0$, we construct a PALF, denoted $SF$, whose total space is diffeomorphic to $D^4$. Finally, by embedding the closed curve $C_0$ in the exact position formerly occupied by the guide line $B_0$, we obtain the desired PALF $P$. 

\smallskip
We now analyze the boundary components of the regular fiber, which can be classified into two distinct types:

\begin{itemize}
\item \textbf{Boundaries of holes formed by pushing the $0$-handle boundary to the right inside the gray regions.}  
For instance, consider the hole generated within the gray region $\G{1}$ after attaching a $1$-handle in column~1 and subsequently pushing the boundary of the $0$-handle to the right via isotopy. Let $jL$ (respectively, $jR$) denote the left (respectively, right) boundary segment of the long vertical band of the $1$-handle located in the $j$-th column. In our specific example, the segments $1R$, $2L$, $4L$, and $8L$ collectively form a single continuous boundary component within $\G{1}$, corresponding to the hole created by the extruded $0$-handle. Let $\hl$ denote the total number of such hole boundaries generated by pushing out the $0$-handle across all gray regions within the entire regular fiber.

\item \textbf{Boundaries along the outer edges of the gray regions.}  
Each gray region $\G{i}$ ($1 \leq i \leq 4$) contributes exactly one boundary component corresponding to its outer perimeter. Referring to Figure~\ref{fig:Y0306}, the outer boundary of $\G{1}$ is composed of the segments $1L, 3R, 2R, 5R, 4R, 11L, 9R, 8R,$ and $17R$; the boundary of $\G{2}$ comprises $5L$ and $6R$; the boundary of $\G{3}$ comprises $9L, 12R, 11R,$ and $15R$; and the boundary of $\G{4}$ comprises $3L, 6L, 12L, 15L,$ and $17L$. Consequently, the total number of such outer boundary components is exactly equal to the number of gray regions, which we denote by $\gr$.
\end{itemize}

Before proceeding to the discussion on the genus, we define several notations and state combinatorial identities that hold generally, not just for this specific example. 

Let $v$, $e$, and $f$ denote the number of vertices, edges, and faces of the associated planar graph, respectively, where the unbounded outer face is excluded from the count of $f$. Furthermore, let $\gr$ denote the number of bounded gray regions, $\wh$ the number of bounded white regions (i.e., excluding the unbounded outer region), and $\crs$ the number of crossings of the knot. By Euler's formula for planar graphs, the following relations hold:
\[
v = \gr, \qquad e = \crs, \qquad f = \wh,
\qquad
v - e + f = \gr - \crs + \wh = 1.
\]

Finally, with this preparation, we compute the genus $g$ of the regular fiber.

\begin{itemize}
\item \textbf{Boundary components:} 
The total number of boundary components of the regular fiber, denoted $\bn$, is the sum of the number of gray regions $\gr$ and the number of holes $\hl$ formed inside the gray regions. Thus, $\bn = \gr + \hl$. 
In our running example, since $\gr = 4$ and $\hl = 1$, we have $\bn = 5$.

\item \textbf{Number of $1$-handles:} 
The number of $1$-handles associated with vertical segments with NW corners is exactly $\wh + \hl$. Meanwhile, the number of $1$-handles associated with vertical segments with NE corners is exactly $\crs$. Therefore, the total number of $1$-handles in the regular fiber, denoted $\hn$, is given by $\hn = \wh + \hl + \crs$. 
In our example, substituting $\wh = 4$, $\hl = 1$, and $\crs = 7$ yields $\hn = 12$.

\item \textbf{Calculation of the genus:} 
As established above, the relation $\gr - \crs + \wh = 1$ generally holds, which implies $\crs = \gr + \wh - 1$. Substituting this into the equation for $\hn$ gives $\hn = 2\wh + \gr + \hl - 1$. 
Recall the topological Euler characteristic formula for a surface constructed from a single $0$-handle and $\hn$ $1$-handles, which states that $2 - 2g - \bn = 1 - \hn$. By substituting $\bn = \gr + \hl$ and our expression for $\hn$ into this formula, we deduce $g = \wh$. 

Consequently, the genus of the regular fiber is precisely equal to the number of white regions in the corresponding planar graph. In our specific example, this gives $g = 4$.
\end{itemize}

\smallskip

We now prove Theorem~\ref{thm:alternating}.

\begin{proof}
As illustrated in the preceding example, to prove the theorem, it suffices to establish the following two claims:

\begin{description}
\item[Claim~1]The total number of boundary components $\bn$ of the regular fiber equals the sum of the number of gray regions $\gr$ and the number of holes $\hl$ formed inside the gray regions.
\item[Claim~2]The total number of $1$-handles associated with vertical segments with NW corners equals the sum of the number of white regions $\wh$ and the number of holes $\hl$.
\end{description}

Recall that the number of $1$-handles associated with vertical segments with NE corners is exactly equal to the number of crossings of the knot, denoted $\crs$. Consequently, if Claim~1 and Claim~2 are verified, we can deduce the desired result $g = \wh$ directly from the following system of four equations:
\[
\bn = \gr + \hl, \quad \hn = \wh + \hl + \crs, \quad \gr - \crs + \wh = 1, \quad 2 - 2g - \bn = 1 - \hn.
\]

As defined in Section~\ref{sec:preliminaries}, we refer to the leftmost vertex of the associated planar graph as the \emph{starting point} and the rightmost vertex as the \emph{terminal point}. Based on this orientation, the gray regions can be classified into three distinct types. To complete the proof, it is necessary and sufficient to verify that Claim~1 and Claim~2 hold for each of these three types of gray regions.

\begin{itemize}
\item \textbf{Gray regions containing the starting point.}
In Figure~\ref{fig:Y0304}, the region $\G{1}$ serves as an example of this type. A detailed step-by-step visualization is provided in Figure~\ref{fig:Y0307}: the top row displays the knot in grid position, with the regions shaded according to the established rules and the corresponding planar graph superimposed. In the middle row, the left panel shows the regular fiber and the simple closed curves (vanishing cycles) resulting from the PALF construction, while the right panel explicitly highlights the five boundary components of the regular fiber by coloring each of them distinctly. Finally, the bottom row depicts the regular fiber deformed by an isotopy to clearly visualize these boundary components.

\begin{figure}[htbp]
\centering
\includegraphics[scale=0.55]{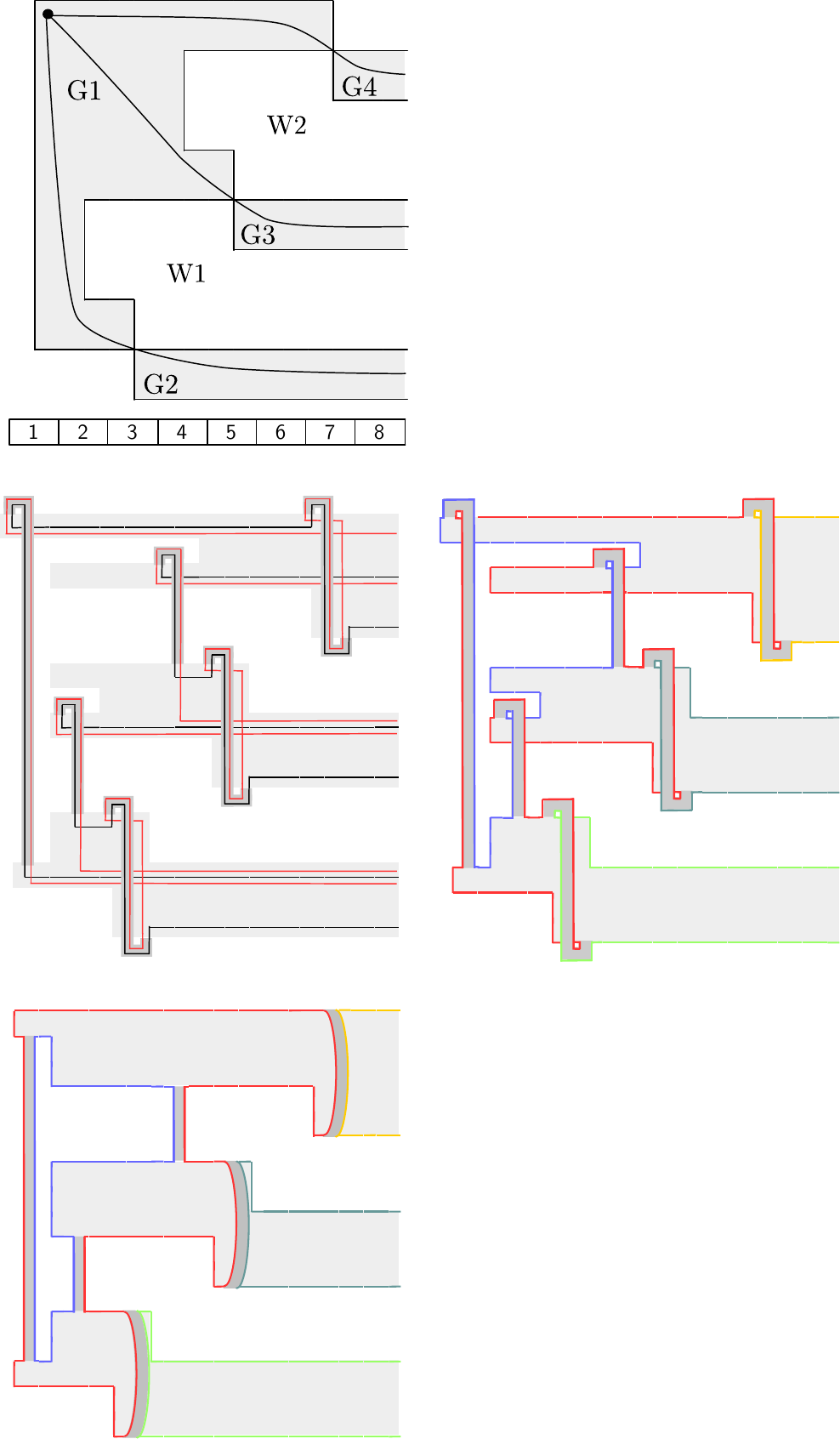}
\caption{The gray region containing the starting point.}
\label{fig:Y0307}
\end{figure}

In the regular fiber, the boundary associated with $\G{1}$ consists of exactly two components: the boundary of the hole formed strictly inside $\G{1}$ (depicted as a blue closed curve) and its outer boundary (depicted as a red closed curve). In contrast, the boundaries associated with $\G{2}$, $\G{3}$, and $\G{4}$ consist solely of their outer boundaries, which are shown as green, dark green, and orange closed curves, respectively.

The distribution of $1$-handles associated with vertical segments with NW corners can be summarized according to our previous classification:
\begin{description}
\item[Case (1)] \textbf{Vertical segments with NW corners located at the left edges of white regions.} 
In this example, the vertical segments with NW corners in columns~2 and~4 are situated at the left edges of the white regions $\W{1}$ and $\W{2}$, respectively.
\item[Case (2)] \textbf{Vertical segments with NW corners other than those in Case (1).} 
The vertical segment with an NW corner in column~1 falls into this category, generating the hole inside $\G{1}$.
\end{description}

Thus, Claim~1 and Claim~2 hold for this type of gray region.

\begin{figure}[htbp]
\begin{minipage}[htbp]{0.48\textwidth}
\centering
\includegraphics[scale=0.71]{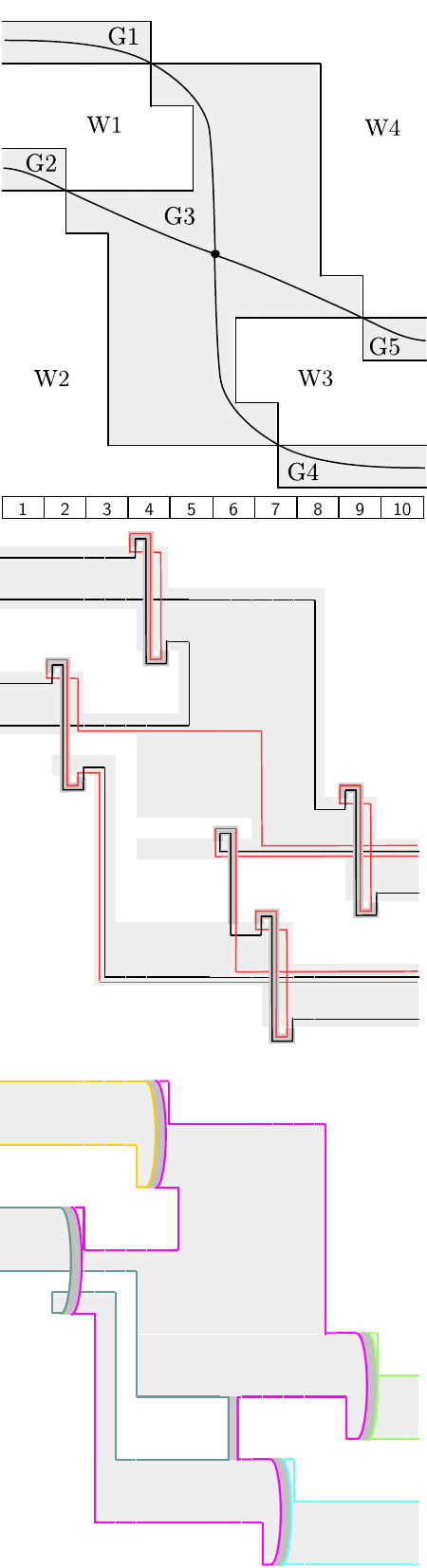}
\caption{Gray regions containing neither the starting point nor the terminal point.}
\label{fig:Y0308}
\end{minipage}
\hfill
\begin{minipage}[htbp]{0.48\textwidth}
\centering
\includegraphics[scale=0.72]{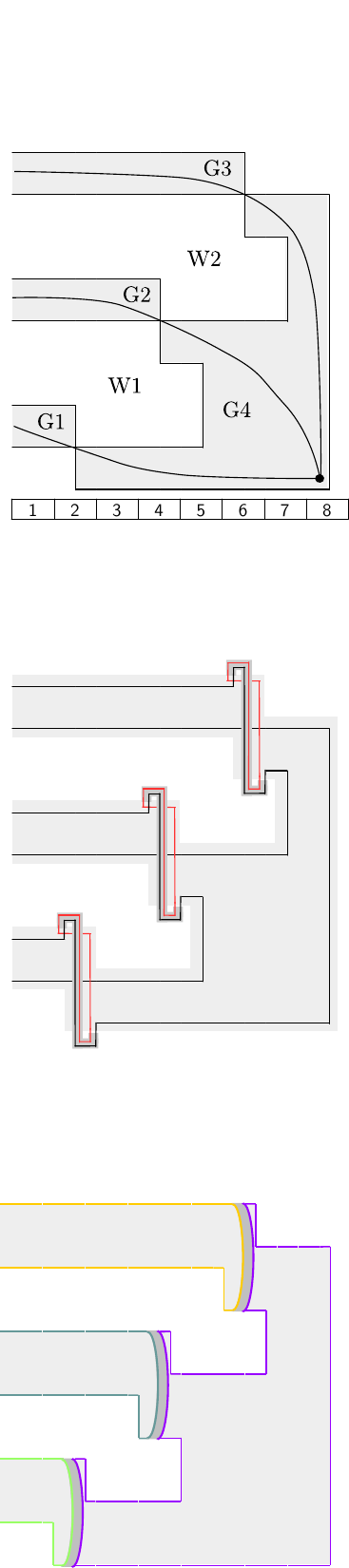}
\caption{The gray region containing the terminal point.}
\label{fig:Y0309}
\end{minipage}
\end{figure}

\item \textbf{Gray regions containing neither a starting point nor a terminal point.}
In Figure~\ref{fig:Y0304}, $\G{2}$ and $\G{3}$ exemplify such regions. This scenario is illustrated in detail in Figure~\ref{fig:Y0308}. Here, we must ensure that the neighborhood of the vertical segment in column~6, located at the left edge of $\W{3}$, becomes part of the boundary of the $0$-handle. Because the vertical segment in column~2 has already been lifted by a $1$-handle, we achieve this by first pushing the boundary of the $0$-handle to the right---sliding it underneath this pre-existing $1$-handle---and subsequently pushing it downward and to the right. Crucially, this isotopy operation does not introduce any new boundary components.

In the regular fiber resulting from the PALF construction, the boundaries of $\G{1}$, $\G{2}$, $\G{3}$, $\G{4}$, and $\G{5}$ correspond precisely to the closed curves colored orange, dark green, magenta, light blue, and green, respectively. Consequently, Claim~1 and Claim~2 hold for this intermediate type of gray region as well.

\item \textbf{Gray regions containing a terminal point.}
In Figure~\ref{fig:Y0304}, $\G{4}$ represents a region of this type. A corresponding example is provided in Figure~\ref{fig:Y0309}. In the regular fiber obtained after the PALF construction, the outer boundaries of $\G{1}$, $\G{2}$, $\G{3}$, and $\G{4}$ correspond to the closed curves colored green, dark green, orange, and purple, respectively. As with the previous cases, Claim~1 and Claim~2 also hold in this final scenario.
\end{itemize}

Since we have verified that Claim~1 and Claim~2 hold for all three possible types of gray regions, the desired equations are established. This completes the proof of Theorem~\ref{thm:alternating}.
\end{proof}

\subsection{Application to positive pretzel knots}\label{subsec:alternating_application}

We now present a concrete application of Theorem~\ref{thm:alternating}. 

\begin{cor}\label{cor:pretzel}
Let $X$ be a Stein surface consisting of a single $0$-handle and $m$ $2$-handles ($m \geq 1$) attached along a Legendrian knot. Suppose the topological type of this knot is a positive pretzel knot $P(q_1, \dots, q_s)$ with $s$ rows, as defined in Definition~\ref{def:pretzel} and illustrated in Figure~\ref{fig:Y0206}. Then, there exists a PALF whose total space is diffeomorphic to $X$, with a regular fiber of genus $s-1$.
\end{cor}

\begin{proof}
By definition, a positive pretzel knot $P(q_1, q_2, \dots, q_s)$ with $s \geq 3$ rows is an alternating knot. As depicted in the standard diagram in Figure~\ref{fig:Y0206}, the associated planar graph for this alternating knot contains exactly $s-1$ bounded white regions. Therefore, by Theorem~\ref{thm:alternating}, there exists a PALF whose total space is diffeomorphic to $X$, and its regular fiber has a genus equal to $s-1$.
\end{proof}

As a specific example, consider the positive pretzel knot $P(2,3,3)$, which corresponds to the knot $8_5$ in the Rolfsen knot table. Following our established procedure, we convert this knot into grid position. We then deform each vertical segment that crosses a horizontal segment into an arc containing a vertical segment of exactly two lattice units in height, bounded by an NE corner and an SW corner. The resulting knot in grid position, denoted $\widetilde{C_0''}$, is depicted in Figure~\ref{fig:Y0310} alongside its associated planar graph.  

The constructed PALF, along with its regular fiber and vanishing cycles, is illustrated in Figure~\ref{fig:Y0312}. The regular fiber has a genus of $2$---which exactly equals the number of white regions in the planar graph---and possesses $8$ boundary components. (Note that in the Kirby diagram associated with this PALF, the framing of the attaching circle corresponding to $C_0$ is $0$.)

\begin{figure}[htbp]
  \begin{minipage}[t]{0.49\textwidth}
    \centering
         \begin{tikzpicture}
            \node[anchor=south west, inner sep=0] (image) at (0,0)    {\includegraphics[scale=0.4]{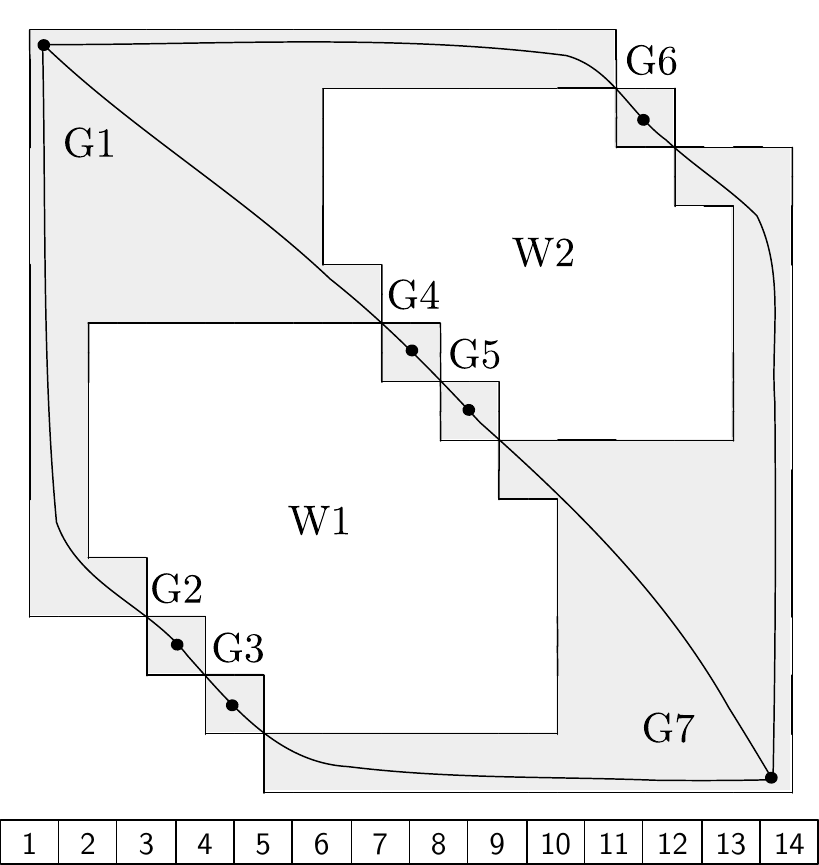}};
            \begin{scope}[x={(image.south east)},y={(image.north west)}]
            \node[font=\scriptsize] at (0, 0.87) {$\widetilde{C_0''}$};
            \end{scope}
         \end{tikzpicture}
    \caption{A positive pretzel knot in grid position and its associated planar graph.}
     \label{fig:Y0310}
   \end{minipage}
  \hfill
  \begin{minipage}[t]{0.49\textwidth}
    \centering
       \begin{tikzpicture}
            \node[anchor=south west, inner sep=0] (image) at (0,0)    {\includegraphics[scale=0.4]{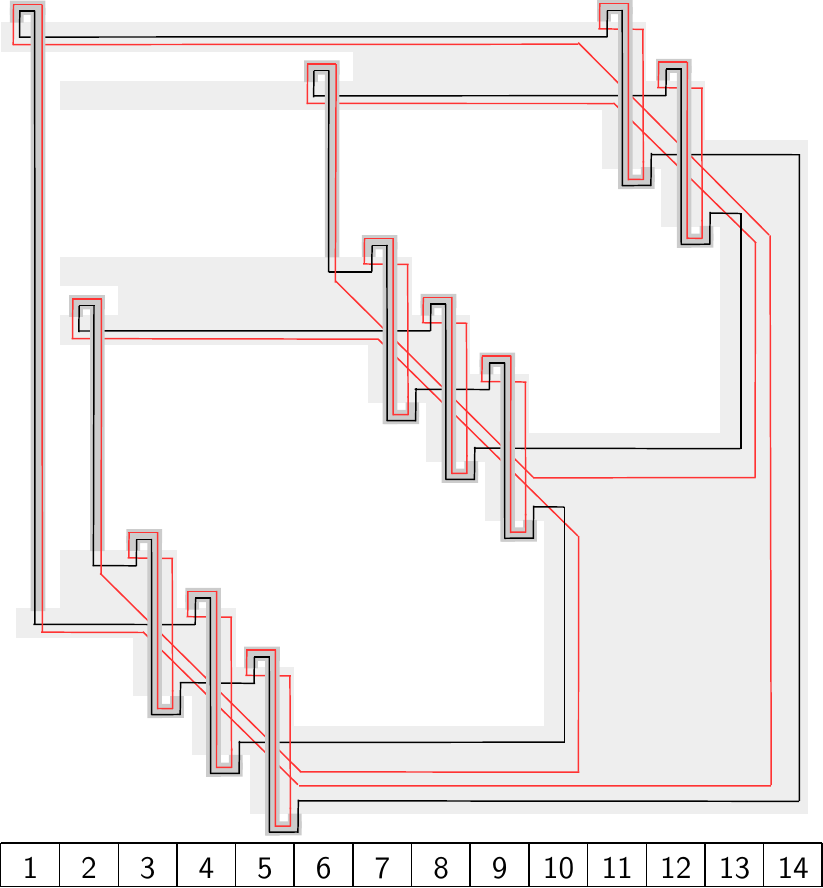}};
            \begin{scope}[x={(image.south east)},y={(image.north west)}]
            \node[font=\scriptsize] at (-0.001, 0.87) {$C_0$};
            \end{scope}
         \end{tikzpicture}
    \caption{Monodromy factorization:
$(C_0, C_{12}, C_{11}, C_9, C_8, C_7, C_6, $ $C_5, C_4,$
$C_3, C_2, C_1)$.}
    \label{fig:Y0312}
  \end{minipage}
\end{figure}

\section{Lefschetz fibrations on knot traces of extended alternating knots}\label{sec:extended}

\subsection{PALF construction for extended alternating knots}\label{subsec:extended_main}

We first present a geometric description of the alternating knots that satisfy Condition~$\mathrm{M}$, as introduced in Section~\ref{sec:preliminaries}. Figure~\ref{fig:Y0401} illustrates the planar diagram of an alternating knot derived from the one previously shown in Figure~\ref{fig:Y0303}. In this representation, the crossings are positioned precisely at the lattice points of a two-dimensional grid, and the associated gray-and-white checkerboard coloring is explicitly indicated. 

Motivated by this geometric representation, we now formally define the broader class of \emph{extended alternating knots}, a class that naturally encompasses the alternating knots of type (a) satisfying Condition~$\mathrm{M}$.

\begin{dfn}\label{def:extended}
An \emph{extended alternating knot} is defined as a knot that admits an \emph{extended alternating diagram}, constructed as follows. 

On a two-dimensional lattice, we place square \emph{crossing nodes} at the lattice points, where each node represents a localized region of crossings. These nodes are then connected by the strands of the knot. The schematic representation of a crossing node is depicted on the left side of Figure~\ref{fig:Y0403}, while its detailed internal structure is illustrated on the right. 

For each crossing node, the parameter $q$ satisfies $q \geq 1$. Furthermore, the total number of strands entering the node from the left must equal the total number of strands exiting to the right, and this common value must be strictly greater than one: $\sum_{i=1}^u p_i = \sum_{j=1}^v r_j > 1$.

The plane containing this diagram is naturally partitioned into four types of areas: the crossing nodes themselves, bounded gray regions, bounded white regions, and the unbounded outer region. Specifically, a \emph{white region} is bounded by the top or bottom edges of the square crossing nodes and the connecting strands of the knot. Conversely, a \emph{gray region} is bounded by the left or right edges of these square nodes and the strands.

Finally, we define a specific reduction operation on this diagram: suppose we retain only the strands that form the boundaries of the gray regions, deleting all other strands, and we replace every crossing node with exactly $q$ positive right-handed half-twists involving two strands. The knot obtained through this operation is formally defined as \emph{the alternating knot corresponding to the extended alternating knot}. We then define the planar graph of the extended alternating knot to be the planar graph of this corresponding alternating knot.
\end{dfn}

Figure~\ref{fig:Y0402} illustrates a concrete example of an extended alternating knot. The corresponding alternating knot, obtained via the reduction operation described above, is depicted in Figure~\ref{fig:Y0401}. 

Although this particular example is arranged in two columns and five rows (with the left column containing four crossing nodes and the right column containing two), there are generally no restrictions on the number of columns or rows in an extended alternating diagram. Furthermore, it is important to note that strands are permitted to pass entirely through the bounded gray regions; for instance, the second vertical strand from the left in this diagram is exactly such a strand. 

With these geometric preparations in place, we are now in a position to state the main theorem of this section.

\begin{figure}[htbp]
\begin{minipage}[htbp]{0.49\textwidth}
\centering
\includegraphics[scale=0.36]{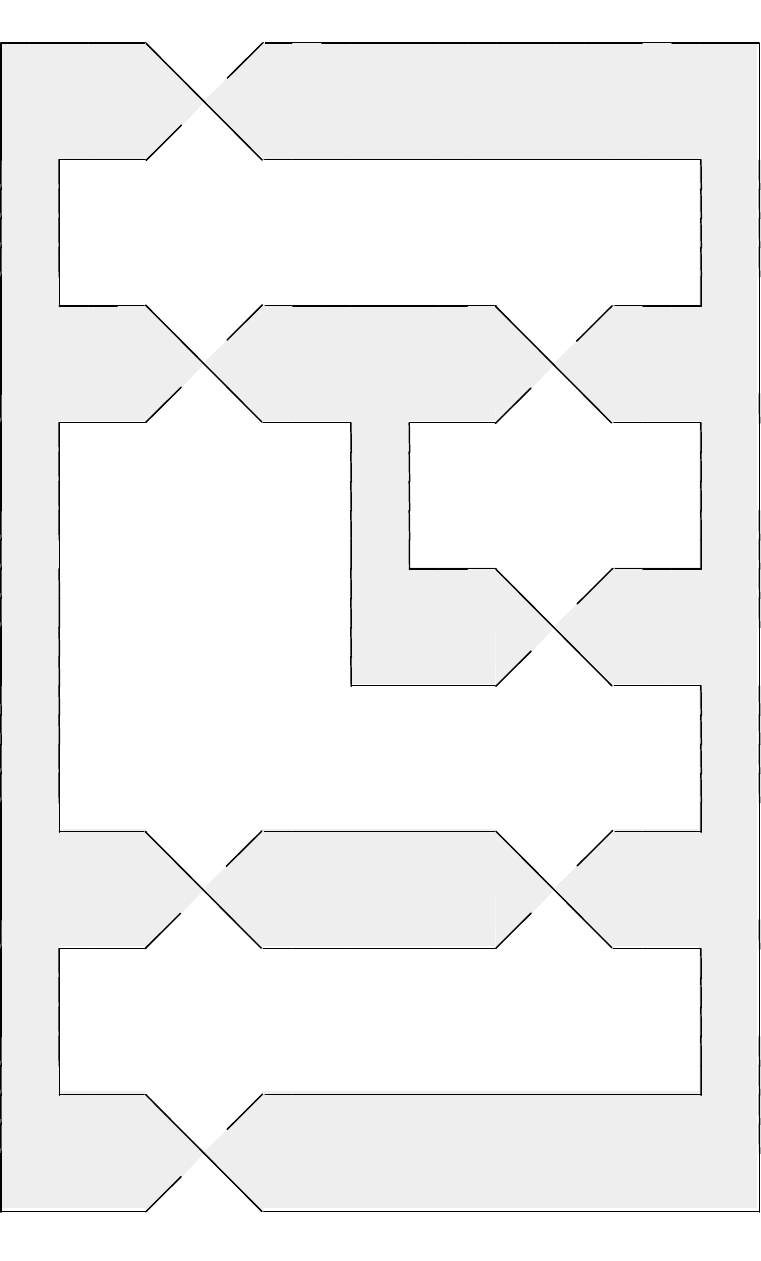}
\caption{The planar diagram of an alternating knot satisfying Condition~$\mathrm{M}$, derived from the one in Figure~\ref{fig:Y0303}.}
\label{fig:Y0401}
\end{minipage}
\hfill
\begin{minipage}[htbp]{0.49\textwidth}
\centering
\includegraphics[scale=0.36]{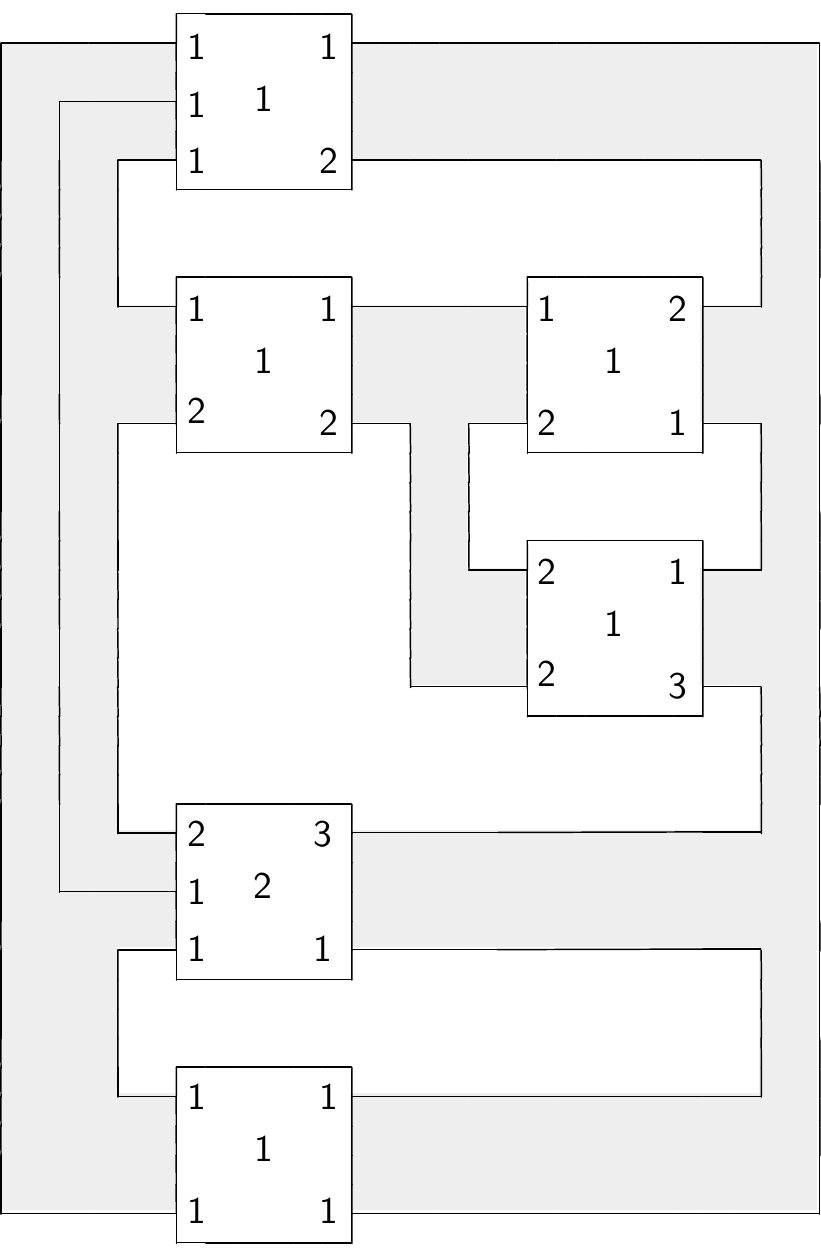}
\caption{An example of an extended alternating knot (its corresponding alternating knot is shown in Figure~\ref{fig:Y0401}).}
\label{fig:Y0402}
\end{minipage}
\end{figure}

\begin{figure}[htbp]
\centering
\includegraphics[scale=0.7]{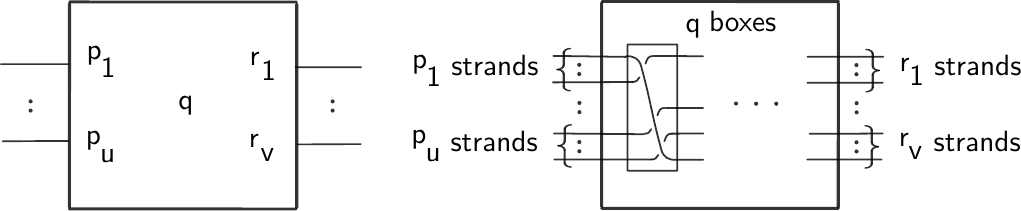}
\caption{A crossing node: schematic representation (left) and its internal structure with $q$ positive crossings (right).}
\label{fig:Y0403}
\end{figure}

\begin{thm}\label{thm:extended}
Let $X$ be a Stein surface consisting of a single $0$-handle and $m$ $2$-handles ($m \geq 1$) attached along a Legendrian knot whose topological type is that of an extended alternating knot. Consider the planar graph associated with this extended alternating knot. 
Then, there exists a positive allowable Lefschetz fibration (PALF) whose total space is diffeomorphic to $X$, and whose regular fiber has a genus exactly equal to the number of white regions in the corresponding planar graph.
\end{thm}

\begin{proof}
Let us consider a concrete example. Figure~\ref{fig:Y0404} illustrates the result of converting the extended alternating knot from Figure~\ref{fig:Y0402} into the knots $\widetilde{C_{0k}''}$ ($1 \leq k \leq 3$) in grid position, alongside the associated planar graph. The guide lines $B_{0k}$ are copies of the corresponding $\widetilde{C_{0k}''}$ drawn on the $0$-handle $D^2$, which is depicted as a gray square (see Figure~\ref{fig:Y0405}). The constructed PALF, along with its regular fiber and vanishing cycles, is shown in Figure~\ref{fig:Y0406}. (Note that in the Kirby diagram associated with this PALF, the framings of the attaching circles corresponding to $\overline{C_{01}}$, $\overline{C_{02}}$, and $\overline{C_{03}}$ are $-3$, $-6$, and $-2$, respectively.)

\begin{figure}[htbp]
  \begin{minipage}[t]{0.49\textwidth}
    \centering
         \begin{tikzpicture}
            \node[anchor=south west, inner sep=0] (image) at (0,0)    {\includegraphics[scale=0.30]{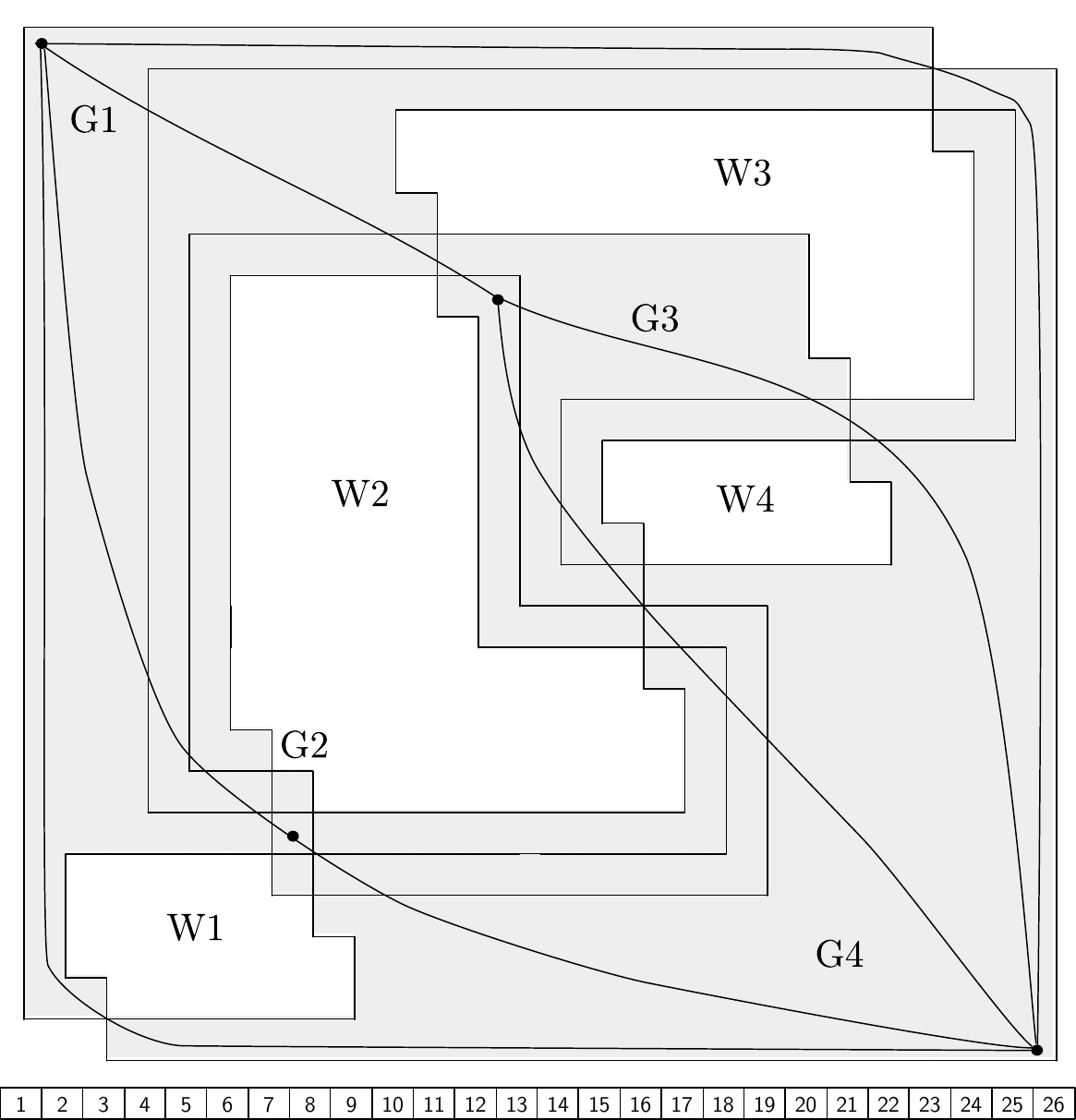}};
            \begin{scope}[x={(image.south east)},y={(image.north west)}]
            \node[font=\scriptsize] at (-0.02, 0.87) {$\widetilde{C_{01}''}$};
            \node[font=\scriptsize] at (0.10, 0.27) {$\widetilde{C_{02}''}$};
            \node[font=\scriptsize] at (0.265, 0.7) {$\widetilde{C_{03}''}$};
            \end{scope}
         \end{tikzpicture}
    \caption{An extended alternating knot in grid position and its associated planar graph.}
    \label{fig:Y0404}
  \end{minipage}
  \hfill
  \begin{minipage}[t]{0.49\textwidth}
    \centering
       \begin{tikzpicture}
            \node[anchor=south west, inner sep=0] (image) at (0,0)    {\includegraphics[scale=0.30]{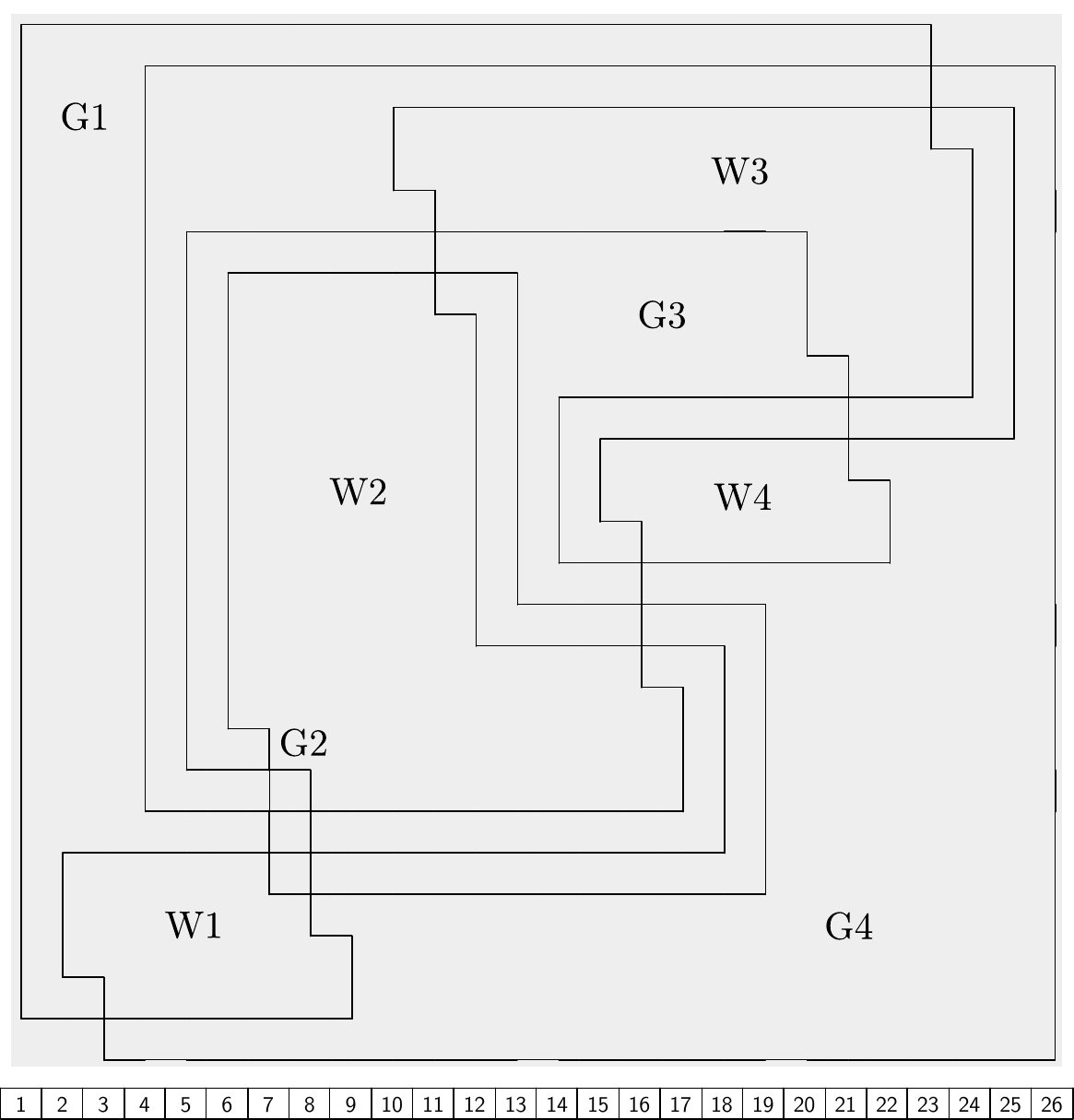}};
            \begin{scope}[x={(image.south east)},y={(image.north west)}]
            \node[font=\scriptsize] at (-0.02, 0.87) {$B_{01}$};
            \node[font=\scriptsize] at (0.095, 0.265) {$B_{02}$};
            \node[font=\scriptsize] at (0.265, 0.7) {$B_{03}$};
            \end{scope}
         \end{tikzpicture}
    \caption{The guide lines $B_{01}$, $B_{02}$, and $B_{03}$ on the surface of the $0$-handle.}
    \label{fig:Y0405}
  \end{minipage}
\end{figure}

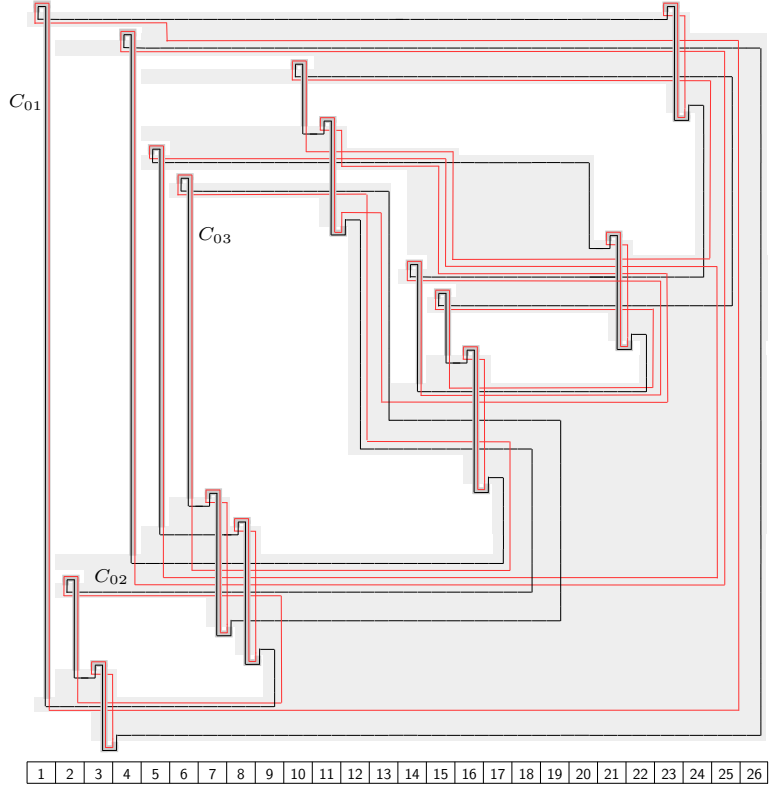
\begin{figure}[htbp]
\centering  
\begin{tikzpicture}
    \node[anchor=south west, inner sep=0] (image) at (0,0)  {\includegraphics[scale=0.54]{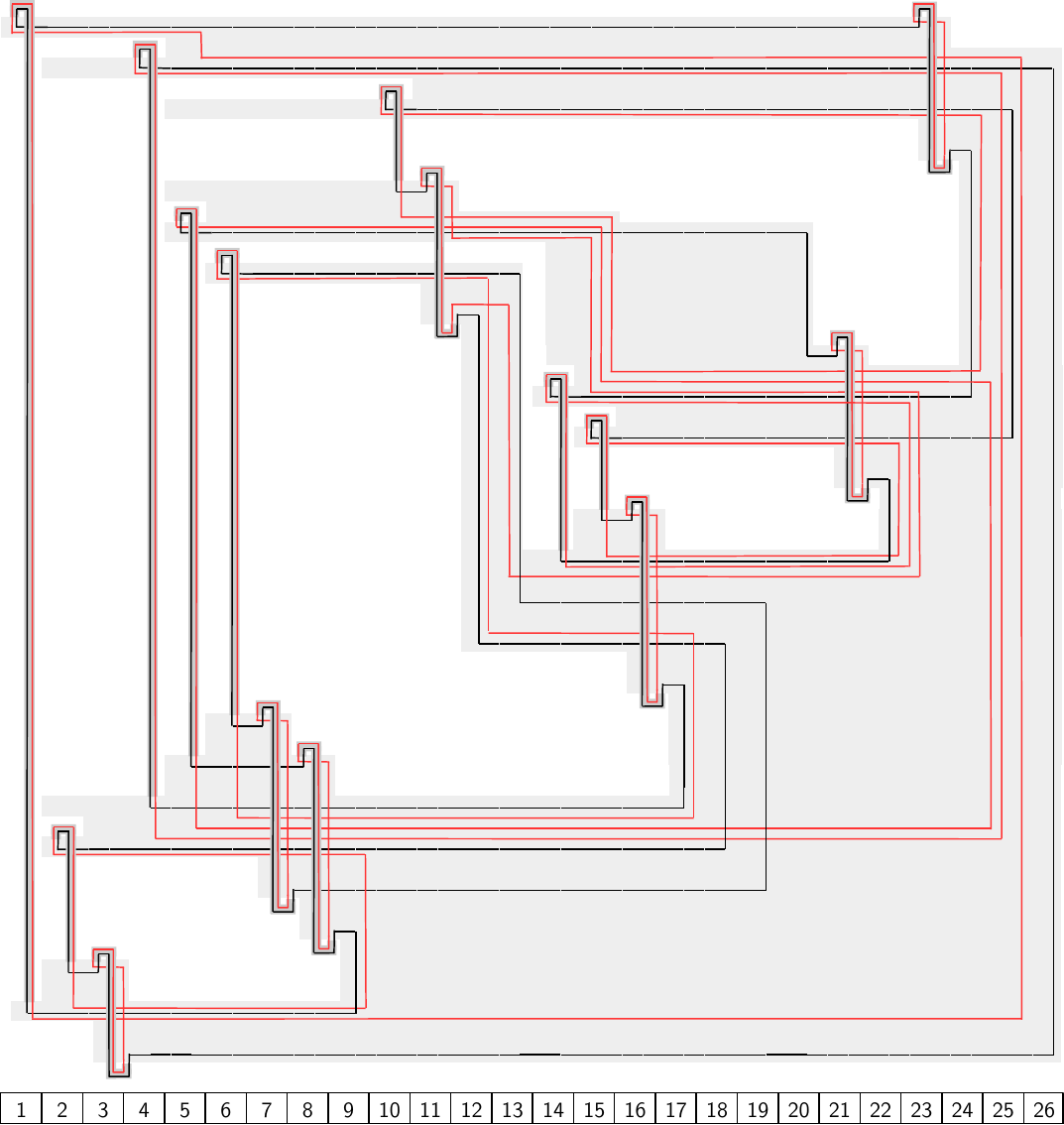}};
    \begin{scope}[x={(image.south east)},y={(image.north west)}]
        \node[font=\scriptsize] at (0, 0.87) {$C_{01}$};
        \node[font=\scriptsize] at (0.115, 0.265) {$C_{02}$};
        \node[font=\scriptsize] at (0.255, 0.7) {$C_{03}$};
    \end{scope}
\end{tikzpicture}
\caption{Monodromy factorization:
$(C_{03}, C_{02}, C_{01}, C_{23}, C_{21}, C_{16}, C_{15}, C_{14}, $
$C_{11}, C_{10}, C_8, C_7, C_6, C_5, C_4, C_3, C_2, C_1)$. Each $C_i$ $(1 \leq i \leq 23)$ denotes a red simple closed curve passing over the $1$-handle in the $i$-th column.}
\label{fig:Y0406}
\end{figure}

The vertical segments of $B_{0k}$, along with the $1$-handles attached to the boundary of the $0$-handle to lift them, can be classified into the following four distinct types:

\begin{description}
\item[Case (1)]  \textbf{Vertical segments with NW corners located at the left edges of white regions.}  
In an extended alternating knot, the total number of such segments is identical to that in its corresponding alternating knot.

\item[Case (2)]  \textbf{Vertical segments with NW corners other than those in Case (1).}  
In our running example, such segments appear in columns~1, 4, and~5 within $\G{1}$, as well as in column~14 within $\G{3}$. In an extended alternating knot, the number of these segments is greater than or equal to that in its corresponding alternating knot.

\item[Case (3)]  \textbf{Vertical segments with NE corners that cross horizontal segments.}  
In an extended alternating knot, the total number of such segments is exactly equal to that in its corresponding alternating knot, even though some of these segments may cross multiple horizontal segments.

\item[Case (4)]  \textbf{Vertical segments with NE corners that do not cross horizontal segments.}  
As established for Case~(4) in Section~\ref{sec:alternating}, no $1$-handles are attached to these segments.
\end{description}

\smallskip
We now analyze the boundary components of the regular fiber, which can be classified into two distinct types:

\begin{itemize}
\item \textbf{Boundaries of holes formed by pushing the boundary of the $0$-handle to the right inside the gray regions.}  
These holes are generated after attaching $1$-handles to lift the vertical segments classified under Case~(2). Consequently, there is a strict one-to-one correspondence between each Case~(2) segment and the boundary of such a hole. Recall that $jL$ (respectively, $jR$) denotes the left (respectively, right) boundary segment of the long vertical band of the $1$-handle located in the $j$-th column. In our specific example illustrated in Figure~\ref{fig:Y0406}, the explicit correspondences are as follows: the segment in column~1 corresponds to a hole containing the segments $1R$, $2L$, and $4L$; the segment in column~4 corresponds to a hole containing $4R$, $5L$, and $10L$; the segment in column~5 corresponds to a hole containing $5R$, $6L$, and $14L$; and the segment in column~14 corresponds to a hole containing $14R$ and $15L$. Let $\hl$ denote the total number of such hole boundary components, which is precisely equal to the number of Case~(2) vertical segments.

\item \textbf{Boundaries along the outer edges of the gray regions.}  
Each gray region $\G{i}$ ($1 \leq i \leq 4$) contributes exactly one boundary component corresponding to its outer perimeter. Thus, the total number of such outer boundary components is exactly equal to the number of gray regions, which we again denote by $\gr$.
\end{itemize}

Just as in the proof of Theorem~\ref{thm:alternating}, the following two claims hold:

\begin{description}
\item[Claim~1]  The total number of boundary components $\bn$ of the regular fiber equals the sum of the number of gray regions $\gr$ and the number of holes $\hl$ formed inside the gray regions.
\item[Claim~2]  The total number of $1$-handles associated with vertical segments with NW corners equals the sum of the number of white regions $\wh$ and the number of holes $\hl$.
\end{description}

By computing the genus $g$ of the regular fiber in the exact same manner as in Theorem~\ref{thm:alternating}, we again obtain $g = \wh$. In our running example, the regular fiber therefore has a genus of $4$ and possesses $8$ boundary components.

The fundamental reason why the genus of the regular fiber associated with the extended alternating knot coincides with that of its corresponding alternating knot can be understood as follows. When transitioning from the corresponding alternating knot to the extended alternating knot, the number of vertical segments increases \emph{exclusively} within the Case~(2) category. Consequently, the total number of boundary components of the regular fiber increases by the exact same amount, due solely to the new holes formed after attaching $1$-handles to lift these additional Case~(2) vertical segments. In the Euler characteristic computation for the genus, these two simultaneous, equal increases perfectly cancel each other out.
\end{proof}

\subsection{Application to positive torus knots and torus pretzel knots}\label{subsec:extended_application}

We now present concrete applications of Theorem~\ref{thm:extended}.

\begin{cor}\label{cor:torus}
Let $X$ be a Stein surface consisting of a single $0$-handle and $m$ $2$-handles ($m \geq 1$) attached along a Legendrian knot. Suppose the topological type of this knot is a positive torus knot $T_{p,q}$, as defined in Definition~\ref{def:torus} and illustrated in Figure~\ref{fig:Y0207}. Then, there exists a PALF whose total space is diffeomorphic to $X$, with a regular fiber of genus $1$.
\end{cor}

\begin{proof}
By definition, a positive torus knot $T_{p,q}$ (where $p > 1$ and $q > 1$) can be represented as an extended alternating knot. Its corresponding extended alternating diagram is depicted in Figure~\ref{fig:Y0407}, where $q = q_1 + q_2$ (with $q_1 \geq 1$ and $q_2 \geq 1$). In this standard diagram, there is exactly one bounded white region. Therefore, by Theorem~\ref{thm:extended}, there exists a PALF whose regular fiber has a genus equal to 1.
\end{proof}

As a specific example, consider the positive torus knot $T_{3,4}$. For its extended alternating diagram, we set the parameters to $p=2$, $q_1=2$, and $q_2=2$. Following our established procedure, we convert this knot into grid position. 
The constructed PALF, along with its regular fiber and vanishing cycles, is illustrated in Figure~\ref{fig:Y0411}. The regular fiber has a genus of $1$---which exactly equals the number of white regions in the planar graph---and possesses $6$ boundary components. (Note that in the Kirby diagram associated with this PALF, the framing of the attaching circle corresponding to $C_0$ is $4$.)

\begin{figure}[htbp]
  \begin{minipage}[t]{0.39\textwidth}
    \centering
         \begin{tikzpicture}[baseline=-0.9cm]
            \node[anchor=south west, inner sep=0] (image) at (0,0)    {\includegraphics[scale=0.5]{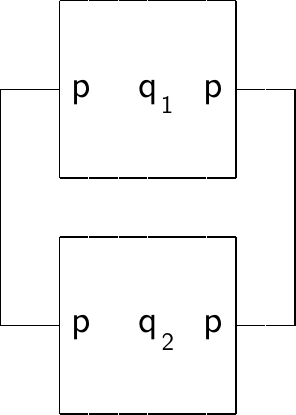}};
            \begin{scope}[x={(image.south east)},y={(image.north west)}]
            \end{scope}
         \end{tikzpicture}
    \caption{The extended alternating diagram of a positive torus knot $T_{p,q}$, where $q = q_1 + q_2$.}
    \label{fig:Y0407}
  \end{minipage}
  \hfill
  \begin{minipage}[t]{0.59\textwidth}
    \centering
       \begin{tikzpicture}
            \node[anchor=south west, inner sep=0] (image) at (0,0)    {\includegraphics[scale=0.5]{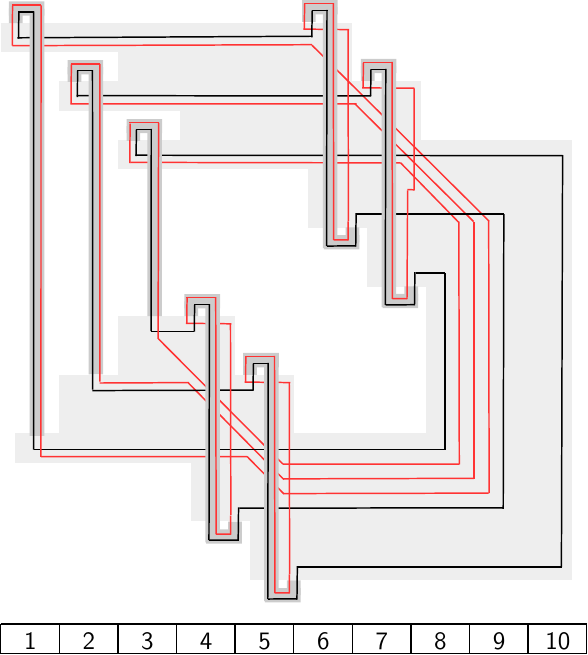}};
            \begin{scope}[x={(image.south east)},y={(image.north west)}]
            \node[font=\scriptsize] at (0, 0.8) {$\widetilde{C_0'}$};
            \end{scope}
         \end{tikzpicture}
    \caption{Monodromy factorization: $(C_0, C_7, C_6, C_5,C_4,C_3,C_2,C_1)$. }
    \label{fig:Y0411}
  \end{minipage}
\end{figure}

Finally, we consider a class of knots obtained by taking positive pretzel knots and replacing each row with the crossings of a positive torus knot. In this paper, we introduce the term \emph{positive torus-pretzel knots} to describe them.

\begin{figure}[htbp]
\begin{minipage}[htbp]{0.4\textwidth}
\centering
\includegraphics[scale=0.4]{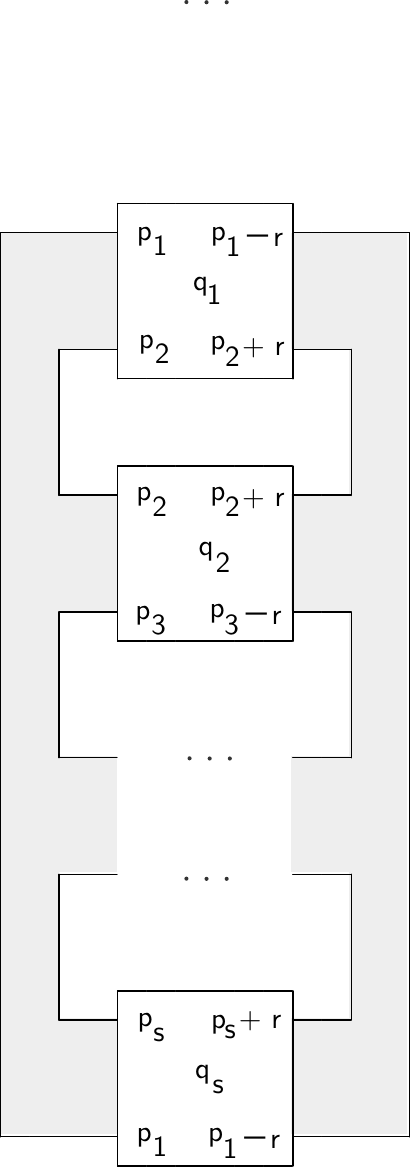}
\caption{The definition of the positive torus-pretzel knot.}
\label{fig:Y0412}
\end{minipage}
\hfill
\begin{minipage}[htbp]{0.58\textwidth}
\centering
\includegraphics[scale=0.62]{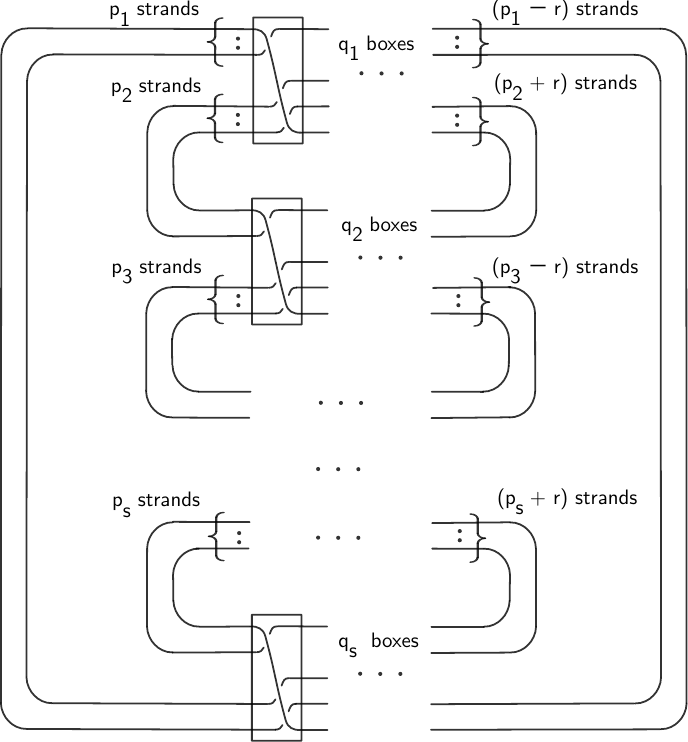}
\caption{A diagram of the positive torus-pretzel knot.}
\label{fig:Y0413}
\end{minipage}
\end{figure}

\begin{figure}[htbp]
\begin{minipage}[htbp]{0.34\textwidth}
\centering
\includegraphics[scale=0.5]{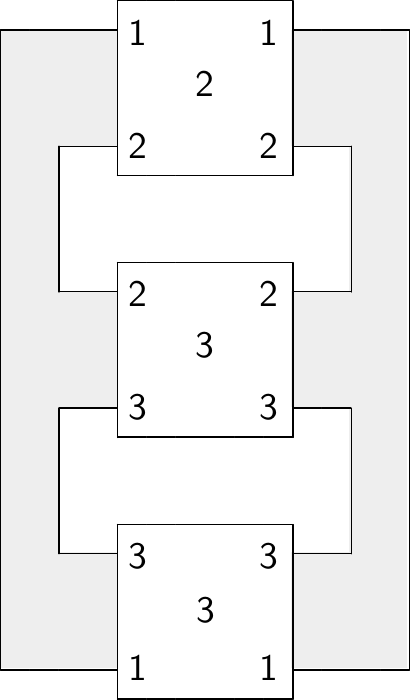}
\caption{An extended alternating diagram of the positive torus-pretzel knot.}
\label{fig:Y0414}
\end{minipage}
\hfill
\begin{minipage}[htbp]{0.64\textwidth}
\centering
\includegraphics[scale=0.37]{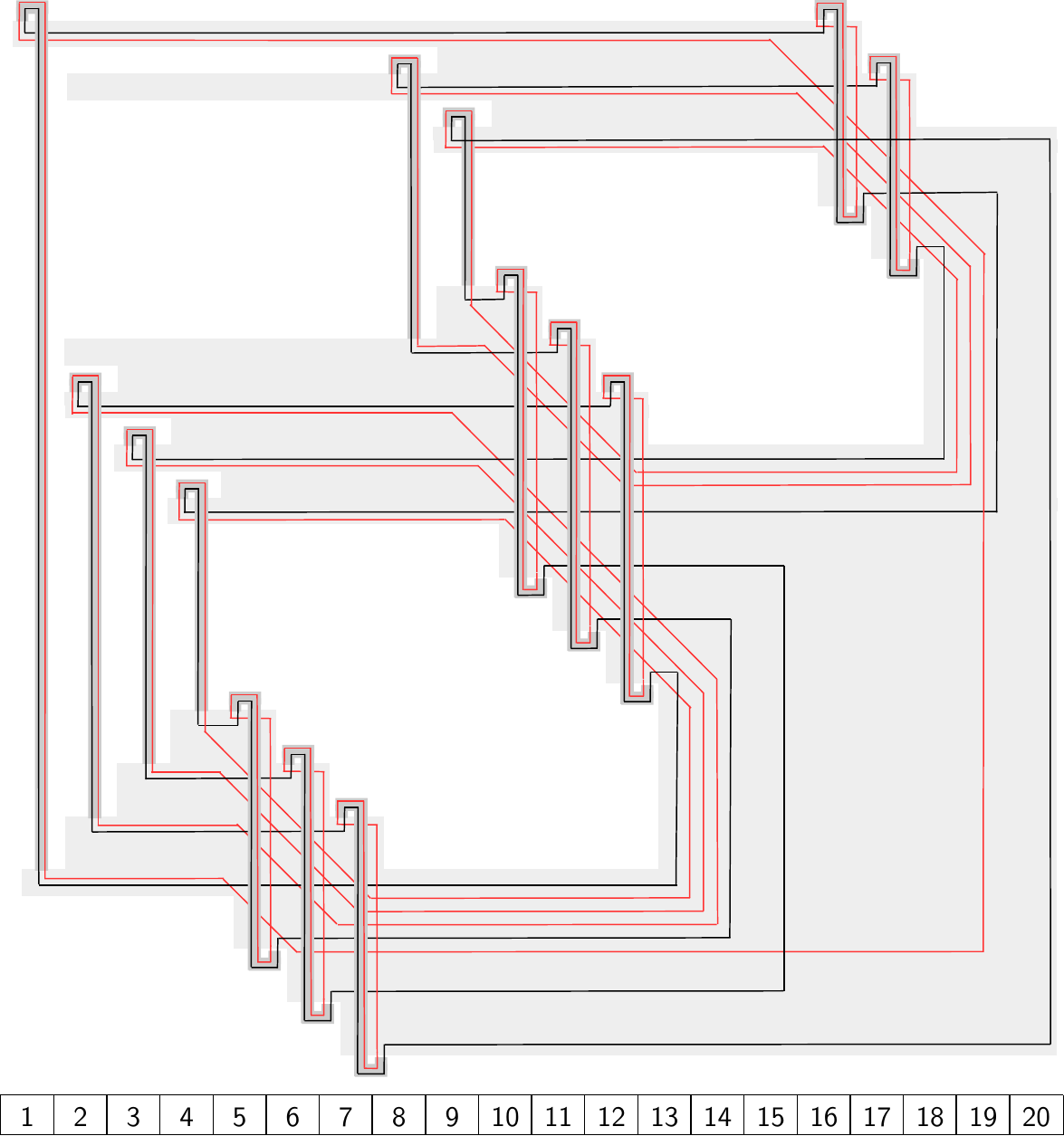}
\caption{Monodromy factorization:
$(C_0, C_{17}, C_{16}, C_{12}, C_{11}, C_{10}, C_9, C_8, C_7, C_6, C_5, $ $C_4, C_3, C_2, C_1)$. }
\label{fig:Y0418}
\end{minipage}
\end{figure}

\begin{dfn}\label{def:torus_pretzel}
A \emph{positive torus-pretzel knot}, denoted $Q(p_1, p_2, \dots, p_s; r; q_1, q_2, \dots, q_s)$, is a knot represented by the extended alternating diagram shown in Figure~\ref{fig:Y0412} (or equivalently, Figure~\ref{fig:Y0413}), provided that the following conditions are satisfied:
\begin{itemize}
\item $s \geq 3$, \quad $p_k \geq 1$, \quad $q_k \geq 1$ \quad $(1 \leq k \leq s)$,
\item if $s$ is odd: \quad $r = 0$,
\item if $s$ is even: \quad $1 - \min \{ p_2, p_4, \dots, p_s \} \leq r \leq \min \{ p_1, p_3, \dots, p_{s-1} \} - 1$.
\end{itemize}
\end{dfn}

The third condition is strictly required to ensure that the values $p_1 - r, p_2 + r, \dots, p_{s-1} - r,$ and $p_s + r$ are all greater than or equal to $1$. Geometrically, each row corresponds to the crossings of a positive torus knot $T_{p_k + p_{k+1}, q_k}$ (where $1 \leq k \leq s$, and we adopt the cyclic convention $p_{s+1} = p_1$).

With this definition in place, we obtain the following corollary.

\begin{cor}\label{cor:torus-pretzel}
Let $X$ be a Stein surface consisting of a single $0$-handle and $m$ $2$-handles ($m \geq 1$) attached along a Legendrian knot. Suppose the topological type of this knot is a positive torus-pretzel knot $Q(p_1, p_2, \dots, p_s; r; q_1, q_2, \dots, q_s)$ with $s$ rows, as defined in Definition~\ref{def:torus_pretzel} and illustrated in Figure~\ref{fig:Y0413}. Then, there exists a PALF whose total space is diffeomorphic to $X$, with a regular fiber of genus $s-1$.
\end{cor}

\begin{proof}
By definition, a positive torus-pretzel knot $Q(p_1, p_2, \dots, p_s; r; q_1, q_2, \dots, q_s)$ with $s$ rows is an extended alternating knot. Its corresponding extended alternating diagram is depicted in Figure~\ref{fig:Y0412}, which clearly exhibits exactly $s-1$ bounded white regions. Therefore, by Theorem~\ref{thm:extended}, there exists a PALF whose regular fiber has a genus equal to $s-1$.
\end{proof}

As a specific example, consider the positive torus-pretzel knot $Q(1,2,3;0;2,3,3)$. Its extended alternating diagram is shown in Figure~\ref{fig:Y0414}. The constructed PALF, along with its regular fiber and vanishing cycles, is illustrated in Figure~\ref{fig:Y0418}. The regular fiber has a genus of 2 and possesses 11 boundary components. (Note that in the Kirby diagram associated with this PALF, the framing of the attaching circle corresponding to $C_0$ is $-6$.)

\section{Modifications for knots with non-maximal Thurston--Bennequin numbers}\label{sec:non_maximal_tb}

As noted in Section~\ref{sec:preliminaries}, when the Thurston--Bennequin (tb) number of the Legendrian knot $\widetilde{C_{0k}}$ ($1 \leq k \leq m$) is strictly less than its maximal value, an additional operation is required. This section details the construction method for this specific scenario.

Let $\mathrm{tb}(\widetilde{C_{0k}})$ denote the tb number of the Legendrian knot $\widetilde{C_{0k}}$, and let $\overline{\mathrm{tb}}(\widetilde{C_{0k}})$ denote its maximal possible tb number. Suppose that $\mathrm{tb}(\widetilde{C_{0k}})$ is strictly less than $\overline{\mathrm{tb}}(\widetilde{C_{0k}})$ by a deficit of $d_k$ (where $d_k \geq 1$). The modification proceeds as follows:

\begin{itemize}
\item First, starting from the Legendrian knot $\widetilde{C_{0k}}$ ($1 \leq k \leq m$), we obtain a knot $\widetilde{C_{0k}'}$ in grid position. This is achieved by applying the Mondrian diagram method described in Section~\ref{sec:preliminaries}, followed by the deformation of any vertical segment crossing a horizontal segment, as illustrated in Figure~\ref{fig:Y0301} of Subsection~\ref{subsec:alternating_main}.
\item For each resulting knot $\widetilde{C_{0k}'}$, the vertical segment located at the maximum column number strictly does not cross any horizontal segments and is bounded by an NE corner and an SE corner. We perform exactly $d_k$ consecutive SE stabilizations at this specific SE corner to obtain the modified knot $\widetilde{C_{0k}''}$.
\end{itemize}

Since each SE stabilization at an SE corner decreases the tb number by exactly one, it naturally follows that $\mathrm{tb}(\widetilde{C_{0k}''})$ is precisely equal to the original target value $\mathrm{tb}(\widetilde{C_{0k}})$.

As a concrete example, consider the Legendrian torus knot $T_{3,3}$ composed of three unknots $\widetilde{C_{0k}}$ ($1 \leq k \leq 3$). The maximal tb number of an unknot is $-1$. Let us assume our specific components have $\mathrm{tb}(\widetilde{C_{01}}) = -3$, $\mathrm{tb}(\widetilde{C_{02}}) = -1$, and $\mathrm{tb}(\widetilde{C_{03}}) = -2$. Consequently, their respective deficits are $d_1 = 2$, $d_2 = 0$, and $d_3 = 1$. The resulting knots $\widetilde{C_{0k}''}$ in grid position and the subsequently constructed PALF are illustrated in Figure~\ref{fig:Y0501}. The regular fiber has a genus of $1$ and possesses $8$ boundary components. (Note that in the Kirby diagram associated with this PALF, the framings of the attaching circles corresponding to $\overline{C_0}$, $\overline{C_1}$, and $\overline{C_2}$ are $-4$, $-2$, and $-3$, respectively.)

\begin{figure}[htbp]
\centering  
\begin{tikzpicture}
    \node[anchor=south west, inner sep=0] (image) at (0,0)  {\includegraphics[scale=0.63]{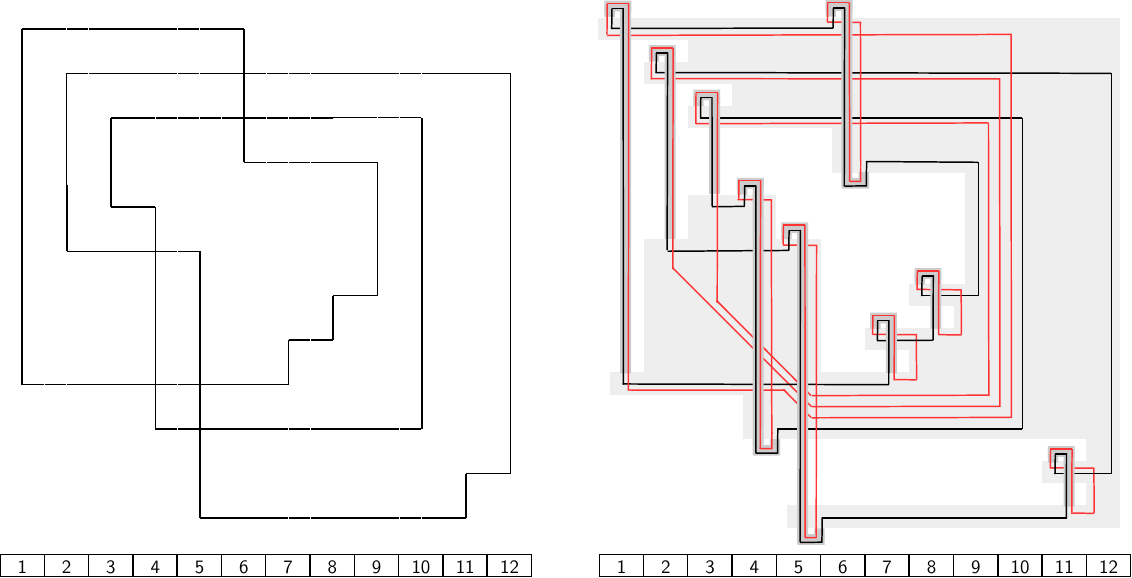}};
    \begin{scope}[x={(image.south east)},y={(image.north west)}]
        \node [below] at (0.23, -0.04) {(a)};
        \node [below] at (0.78, -0.04) {(b)};
        
        \node[font=\scriptsize] at (0.134, 0.98) {$\widetilde{C_{01}''}$};
        \node[font=\scriptsize] at (0.134, 0.905) {$\widetilde{C_{02}''}$};
        \node[font=\scriptsize] at (0.134, 0.83) {$\widetilde{C_{03}''}$};
        \node[font=\scriptsize] at (0.68, 0.98) {$C_{01}$};
        \node[font=\scriptsize] at (0.68, 0.905) {$C_{02}$};
        \node[font=\scriptsize] at (0.68, 0.83) {$C_{03}$};
    \end{scope}
\end{tikzpicture}
\caption{A modification for the Legendrian torus knot $T_{3,3}$ with a non-maximal Thurston--Bennequin number: (b) Monodromy factorization:
($C_{03}, C_{02}, C_{01}, C_{11}, C_8, C_7, C_6, C_5, C_4, C_3, C_2, C_1$).}
\label{fig:Y0501}
\end{figure}
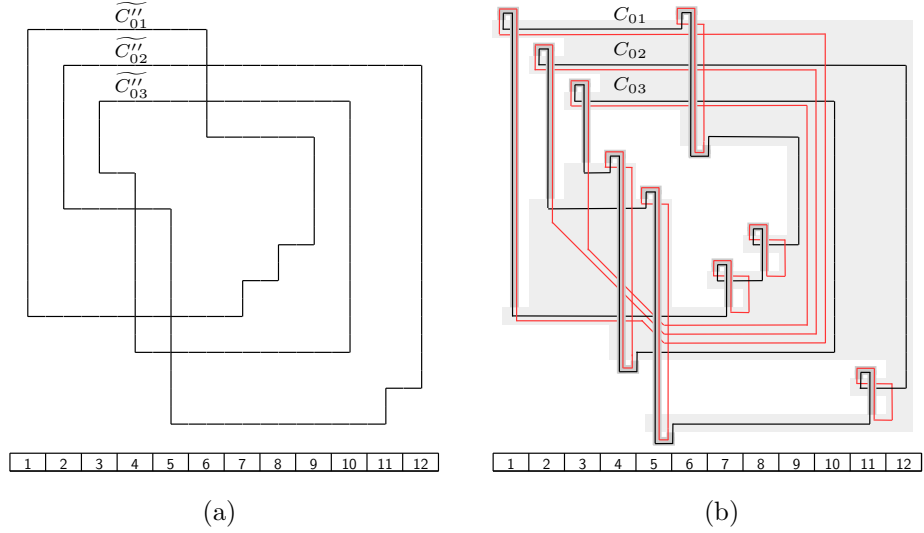

\bibliographystyle{amsalpha}
\bibliography{ALF_1}

\end{document}